\documentclass[10pt]{amsart} 
\usepackage[utf8]{inputenc}
\usepackage[T1]{fontenc}
\usepackage{microtype}
\usepackage{amsmath,amsthm,amsfonts,amssymb}
\usepackage{mathrsfs}
\usepackage{cite}
\usepackage{enumerate}%\usepackage{paralist}
\usepackage{cite}
\usepackage[all,2cell]{xy} \SilentMatrices
\usepackage{pdftricks}
\usepackage{epsfig,color}
\usepackage{mathtools}
\usepackage{tikz-cd}

\usepackage{framed}
\usepackage{color}
\definecolor{shadecolor}{gray}{0.875}
\definecolor{col}{RGB}{42, 95, 151}
\usepackage[colorlinks=true, linkcolor=col, citecolor=col, filecolor = col, menucolor = col, urlcolor = col]{hyperref}

\numberwithin{equation}{section}

\theoremstyle{plain}
\newtheorem{theorem}{Theorem}[section]
\newtheorem*{lemma*}{Lemma}
\newtheorem{lemma}[theorem]{Lemma}
\newtheorem*{theorem*}{Theorem}
\newtheorem{proposition}[theorem]{Proposition}
\newtheorem*{proposition*}{Proposition}
\newtheorem{corollary}[theorem]{Corollary}
\newtheorem*{corollary*}{Corollary}

\theoremstyle{definition}
\newtheorem{remark}[theorem]{Remark}
\newtheorem*{remark*}{Remark}
\newtheorem*{definition*}{Definition}

\newtheorem*{example*}{Example}
\newtheorem{example}[theorem]{Example}

\newtheorem*{question*}{Question}

\def\perv{{\textrm p}}
\def\RR{{\mathbb R}}

\def\ker{\operatorname{ker}}
\def\IC{\operatorname{IC}}

\def\im{\operatorname{im}}

\def\c1{\operatorname{c_1}}
\def\c2{\operatorname{c_2}}

\def\Perv{\operatorname{Perv}}

\def\BM{\mathrm{BM}}
\def\CC{{\mathbb C}}
\def\ZZ{{\mathbb Z}}
\def\NN{{\mathbb N}}
\def\SS{{\mathbb S}}

\def\QQ{{\mathbb Q}}
\def\PP{{\mathbb P}}

\def\A{{\mathcal A}}

\def\O{{\mathcal O}}

\def\E{{\mathscr E}}

\def\+{\oplus}                   
\def\*{\otimes}

\def\ci{{\mathcal{C}^{\infty}}}

\def\Id{\operatorname{Id}}

\def\CH{\operatorname{CH}}
\def\Sq{\operatorname{Sq}}
\def\Wu{\operatorname{Wu}}

\newcommand\blfootnote[1]{%
  \begingroup
  \renewcommand\thefootnote{}\footnote{#1}%
  \addtocounter{footnote}{-1}%
  \endgroup
}

\newcommand{\isoto}{\myxrightarrow{\,\sim\,}}

\makeatletter
\def\myrightarrow{{\setbox\z@\hbox{$\rightarrow$}\dimen0\ht\z@\multiply\dimen0 6\divide\dimen0 10\ht\z@\dimen0\box\z@}}
\def\myrightarrowfill@{\arrowfill@\relbar\relbar\myrightarrow}
\newcommand{\myxrightarrow}[2][]{\ext@arrow 0359\myrightarrowfill@{#1}{#2}}
\def\myleftarrow{{\setbox\z@\hbox{$\leftarrow$}\dimen0\ht\z@\multiply\dimen0 6\divide\dimen0 10\ht\z@\dimen0\box\z@}}
\def\myleftarrowfill@{\arrowfill@\myleftarrow\relbar\relbar}
\newcommand{\myxleftarrow}[2][]{\ext@arrow 3095\myleftarrowfill@{#1}{#2}}
\makeatother

\renewcommand\setminus{-}

\DeclarePairedDelimiter{\floor}{\lfloor}{\rfloor}

\def\X{\mathcal X}

\title{Two coniveau filtrations}

\author{Olivier Benoist}
\address{D\'epartement de math\'ematiques et applications et CNRS, \'Ecole normale sup\'erieure,
45 rue d'Ulm, 75230 Paris Cedex 05, France}
\email{olivier.benoist@ens.fr}

\author{John Christian Ottem}
\address{University of Oslo, Box 1053, Blindern, 0316 Oslo, Norway}
\email{johnco@math.uio.no}

\begin{document}
\maketitle
\date{}
\vspace{-1cm}
\begin{abstract}
A cohomology class of a smooth complex variety of dimension $n$ has coniveau $\geq c$ if it vanishes in the complement of a closed subvariety of codimension $\geq c$, and has strong coniveau $\geq c$ if it comes by proper pushforward from the cohomology of a smooth variety of dimension $\leq n-c$. We show that these two notions differ in general, both for integral classes on smooth projective varieties and for rational classes on smooth open varieties.
\end{abstract}

\thispagestyle{empty}
\def\tors{{\rm tors}}

\blfootnote{{\it Date}: \today}

\section{Introduction}

Let $X$ be a smooth complex variety of dimension $n$. We say that a cohomology class $\alpha\in H^l(X,A)$ with coefficients in an abelian group $A$ has {\em coniveau} $\geq c$ if it vanishes outside a closed subset $Z\subset X$ of codimension at least $c$.
We also say that the class $\alpha$ has {\em strong coniveau} $\geq c$ if it is the Gysin pushforward of a class $\beta\in H^*(Y,A)$ on a smooth variety $Y$ of dimension at most $n-c$ via some proper morphism $f:Y\to X$. These two notions give two filtrations on the cohomology group $H^l(X,A)$, denoted $N^c H^l(X,A)$  and $\widetilde N^c H^l(X,A)$ respectively. Clearly $\widetilde N^c H^l(X,A)\subseteq N^c H^l(X,A)$.

In \cite[\S 9.7]{grothA}, Grothendieck asserted that these two filtrations coincide, i.e., $\widetilde N^cH^l(X,A)=N^cH^l(X,A)$ (in \emph{loc.~cit.}, $X$ is assumed proper and $A$ finite, but these hypotheses are not used in the argument sketched there).
This statement is indeed true if $X$ is proper and $A=\QQ$, as a consequence of Deligne's mixed Hodge theory \cite[Corollaire 8.2.8, Remarque 8.2.9]{HodgeIII}. 
However, a few years later, Grothendieck retracted this statement in a footnote of \cite[p.~300]{grothB} (see also the comments of Illusie in \cite[p.~118]{illusie}).

The goal of this paper is to exhibit the first examples where the two filtrations are indeed different. We give both examples with integral coefficients on smooth projective varieties, and with rational coefficients on smooth open varieties (as well as examples of an appropriate variant of this problem with rational coefficients on singular projective varieties). Here is our first main result:

\begin{theorem}[Theorem \ref{main}]
\label{mainthm}
For all $c\ge 1$ and $l\ge 2c+1$, there is a smooth projective complex variety $X$ such that the inclusion 
 $\widetilde N^cH^l(X,\ZZ)\subset N^c H^l(X,\ZZ)$ is strict.
 One may choose $X$ to have torsion canonical bundle. If $c\geq 2$, one may choose $X$ to be rational.
\end{theorem}

Theorem \ref{mainthm} is optimal as $\widetilde N^cH^l(X,\ZZ)= N^cH^l(X,\ZZ)$ for all $l\leq 2c$ (see Proposition~\ref{egaliteconiveau}); in fact, $N^cH^l(X,\ZZ)=0$ for $l<2c$.
Moreover, $\widetilde N^1H^l(X,\ZZ)=N^1H^l(X,\ZZ)$ if $X$ is rational (see Corollary \ref{rational}). In most of our examples, $H^l(X,\ZZ)$ has torsion, but we also construct one for which $H^l(X,\ZZ)$ is torsion-free (see Proposition~\ref{G2}). 
Our examples are mainly of large dimension, but we also construct some low-dimensional examples.

\begin{theorem}[Theorem \ref{main2}]
\label{mainthm2}
For $l\in\{3,4\}$, there is a smooth projective complex variety $X$ of dimension $l+1$
with torsion canonical bundle
such that the inclusion $\widetilde N^1H^l(X,\ZZ)\subset N^1H^l(X,\ZZ)$ is strict. 
\end{theorem}

The obstructions to the equality $\widetilde N^cH^l(X,\ZZ)=N^cH^l(X,\ZZ)$ that we use to prove Theorems \ref{mainthm} and \ref{mainthm2} are of topological nature, based on Steenrod operations or complex cobordism, and are inspired by the famous examples of Atiyah--Hirzebruch and Totaro of non-algebraic cohomology classes \cite{AH2,totaro}. 
 In fact, we show that in the setting above, some classes in $N^cH^l(X,\ZZ)$ are not even pushforwards from a compact complex manifold of dimension $\leq \dim(X)-c$ via a proper $\mathcal {C}^\infty$\nobreakdash-map.
 The dimensions of the varieties appearing in Theorem~\ref{mainthm2} are the lowest possible that one can obtain with such topological arguments (see Theorem~\ref{main3} and Remark~\ref{rkconiveaus}).

Our second main theorem is:

\begin{theorem}[Theorem \ref{main4}]
\label{mainthm4}
For all $c\geq 1$ and $l\geq 2c+1$, there is a smooth quasi-projective rational complex variety $X$ of dimension $l-c+1$
such that the inclusion $\widetilde N^cH^{l}(X,\QQ)\subset N^cH^{l}(X,\QQ)$ is strict.
\end{theorem}

Theorem \ref{mainthm4} is optimal as $\widetilde N^cH^l(X,\QQ)= N^cH^l(X,\QQ)$ for $l\leq 2c$ (see Proposition~\ref{egaliteconiveau}). Moreover, the dimensions of the varieties we consider are the smallest possible as $\widetilde N^cH^l(X,\QQ)= N^cH^l(X,\QQ)$ if $\dim(X)\leq l-c$ (see Proposition~\ref{egaliteconiveau}).
The proof of Theorem \ref{mainthm4} is based on the theory of perverse sheaves, and relies in an essential way on the decomposition theorem of Bernstein, Beilinson, Deligne and Gabber \cite{BBDG} and on a refinement of the Hodge index theorem  due to de Cataldo and Migliorini \cite{dCM} (see \S\ref{dCM}).
The proof of Theorem \ref{mainthm4} also yields examples demonstrating that the natural coniveau and strong coniveau filtrations on the rational homology of a singular projective  variety may differ (see Theorem~\ref{singular}).

\vspace{1em}

The paper is organized as follows. Section \ref{prelim} gathers generalities on the coniveau and strong coniveau filtrations.
In Section \ref{topo}, we develop topological obstructions for integral cohomology classes to have high strong coniveau. In Sections \ref{BG} and \ref{low} we then give explicit examples showing that these obstructions actually occur, in particular proving Theorems \ref{mainthm} and \ref{mainthm2}. Section \ref{ratcoeffs} deals with cohomology classes with rational coefficients on open or singular varieties and contains the proof of Theorem \ref{mainthm4}. Finally, we collect several questions that we leave open in Section \ref{questions}.

\medskip

\noindent {\bf Conventions}. A variety is a separated scheme of finite type over a field, which will always be the field of complex numbers.   All manifolds are Hausdorff and second countable. All topological spaces have the homotopy type of CW complexes. We use the Grothendieck notation for projective bundles, so that $\PP(\E)$ parameterizes quotient line bundles of a vector bundle $\E$.

\medskip

\noindent {\bf Acknowledgements}. This paper started with a question of Claire Voisin during a visit of the second author to Paris in June 2019. We would like to thank her, as well as Burt Totaro, H\'el\`ene Esnault, Sergey Gorchinsky, Geoffroy Horel, Jørgen Vold Rennemo, Olivier Wittenberg and Nobuaki Yagita for useful discussions. We also thank the referees for their careful work. 
JCO was supported by the Research Council of Norway project no. 250104.

\section{Coniveau and strong coniveau}\label{prelim}

\subsection{Two filtrations}
\label{2filtrations}

Let $X$ be a smooth complex algebraic variety of dimension~$n$. Let us introduce the following two filtrations on the
cohomology of $X$ with coefficients in an abelian group $A$. The first is the classical {\em coniveau filtration}, defined by
\begin{eqnarray*} 
N^c H^l(X,A) & = & \sum_{Z\subset X} \ker \left(j^*:H^l(X,A)\to H^l(X-Z,A)\right)\\
&=&\sum_{Z\subset X} \im \left(H^l_Z(X,A)\to H^l(X,A)\right)
\end{eqnarray*}
where $Z\subset X$ runs through the closed subvarieties of codimension at least $c$ of~$X$ and $j:X-Z\to X$ is the complementary open immersion. 

Similarly, we define the {\em strong coniveau filtration}
$$\widetilde N^c H^l(X,A)=\sum_{f:Y\to X} \im \left(f_*:H^{l-2r}(Y,A)\to H^l(X,A)\right)$$
where the sum is over all proper morphisms $f:Y\to X$ from a smooth complex variety $Y$ of dimension $n-r$ with $r\geq c$. If $X$ is proper, one may equivalently define 
\begin{equation}
\label{alternastrong}
\widetilde N^c H^l(X,A)=\sum_{\Gamma\in\CH^k(Y\times X)} \im \left(\Gamma_*:H^{l-2r}(Y,A)\to H^l(X,A)\right)
\end{equation}
where $Y$ runs over all smooth proper complex varieties of dimension $k-r$ with $r\geq c$, as may be seen by desingularizing the irreducible components of a cycle representing $\Gamma$.

We thus get for each $l$ two descending filtrations $N^c$ and  $\widetilde N^c$ on $H^{l}(X,A)$. We say that a class in $N^c H^l(X,A)$ has {\em coniveau} $\geq c$ and that a class in $\widetilde N^c H^l(X,A)$ has {\em strong coniveau}~$\geq c$. Taking $Z=f(Y)$ shows that $\widetilde N^c H^l(X,A)\subseteq N^c H^l(X,A)$.

Note that we may equivalently define $\widetilde N^c H^l(X,A)$ to be generated by the Gysin pushforwards $i_* \beta$ where $i:\widetilde{Z}\to X$ is a composition of a desingularization $\widetilde Z\to Z$ of a subvariety $Z\subset X$ of codimension $\geq c$ with the inclusion (to see it, introduce a desingularization $\widetilde{Y}\to Y$ admitting a compatible morphism $\widetilde{Y}\to \widetilde{Z}$). 
From this point of view, that the inclusion $\widetilde N^c H^l(X,A)\subseteq N^c H^l(X,A)$ may not be an equality stems from the fact that $\widetilde Z$ and $Z$ can have quite different topology.

We also note that we may restrict, in the above definition of $\widetilde N^c H^l(X,A)$, to morphisms $f:Y\to X$ where $Y$ has dimension $n-c$. Indeed, if $\dim(Y)=n-r$, one may replace $Y$ with $Y\times \PP^{r-c}$ and $f:Y\to X$ with $f\circ pr_1: Y\times \PP^{r-c}\to X$.

One may still define coniveau and strong coniveau filtrations on the Borel--Moore homology of possibly singular varieties. We prefer to stick to the cohomology of smooth varieties for simplicity, except in \S\ref{singularpar}, which is devoted to singular varieties.

\subsection{When coniveau and strong coniveau coincide}

We first recall Deligne's result \cite[Corollaire 8.2.8]{HodgeIII}, whose proof is based on a weight argument.

\begin{theorem}[Deligne]
\label{Deligne}
Let $X$ be a smooth proper complex variety. Then, for all $l,c\geq 0$, one has $\widetilde N^c H^l(X,\QQ)= N^c H^l(X,\QQ)$.
\end{theorem}

We now gather general properties of the coniveau and strong coniveau filtration, valid for any coefficient group $A$.

\begin{proposition}
\label{egaliteconiveau}
Let $X$ be a smooth complex variety of dimension $n$ and let $A$ be an abelian group. If $l\leq 2c$ or if $n\leq l-c$, then $\widetilde N^c H^l(X,A)= N^c H^l(X,A)$.
\end{proposition}

\begin{proof}
Arguing as in \cite[VI, Lemma 9.1 and below]{milne}, we see that $N^c H^l(X,A)=0$ if $l<2c$ and consists of algebraic classes if $l=2c$. 
If $\alpha\in N^cH^{2c}(X,A)$ 
is the class of a subvariety $Z$ of codimension $c$ in $X$ and if $\pi:\widetilde Z\to Z$ is a desingularization of $Z$, then $\alpha$ is the image of $1$ by the Gysin morphism $H^0(\widetilde Z,A)\to H^{2c}(X,A)$.
This proves the first assertion.

To prove the second assertion, we may assume that $X$ is quasi-projective by Chow's lemma. Let $Z\subset X$ be the intersection of $X$ with a general codimension $c$ linear space in some projective embedding. Then $Z$ is smooth by the Bertini theorem and the Gysin morphism $H^{l-2c}(Z,A)\to H^l(X,A)$ is surjective by Hamm's Lefschetz theorem \cite[Theorem 2 and Remark below]{Hamm} applied $c$ times.
%other ref = [Goresky MacPherson, Stratified Morse theory, Theorem p. 153]
\end{proof}

\begin{lemma}
\label{timesPn}
The group $N^1H^l(X,A)/\widetilde N^1H^l(X,A)$ is invariant under replacing $X$ with $X\times \PP^n$ for all $l\geq 0$ and all abelian groups~$A$.
\end{lemma}

\begin{proof}
Let $\pi : X\times \PP^n\to X$ denote the first projection. Using the K\"unneth theorem, we see that $N^1H^l(X\times \PP^n,A)=\pi^*N^1H^l(X,A) \mod \widetilde N^1H^l(X\times \PP^n,A)$. So it suffices to show that a class $\alpha\in H^l(X,A)$ has strong coniveau $\ge 1$ if and only if $\pi^*\alpha$ does. The `only if' direction is clear. Conversely, suppose that $\pi^*\alpha$ has strong coniveau $\ge 1$. We may assume that $\pi^*\alpha=f_*\beta$, where $f:V\to X\times \PP^n$ is a proper morphism from a smooth variety of dimension $\dim X+n-1$, and $\beta\in H^{l-2}(V,A)$. Let $i: X\to X\times \PP^n$ denote the inclusion of a general fiber $X\times \{p\}$. Then $W=V\times_{X\times \PP^n}X$ is a smooth subvariety of $V$ of dimension $\dim X-1$, by Bertini's theorem. Let $g: W\to X$ be the induced map. Then $g_*(\beta|_W)=i^*f_*\beta=i^*\pi^*\alpha=\alpha$, so $\alpha$ has strong coniveau~$\geq$~1.
\end{proof}

\begin{proposition}
\label{birinv}
The group $N^1H^l(X,A)/\widetilde N^1H^l(X,A)$ is a stable birational invariant of smooth projective complex varieties for all $l\geq 0$ and all abelian groups~$A$.
\end{proposition}

\begin{proof}
By Lemma \ref{timesPn}, we only need to show the birational invariance of the group $N^1H^l(X,A)/\widetilde N^1H^l(X,A)$.

The action $\Gamma_*:H^l(Y,A)\to H^l(X,A)$ of a correspondence $\Gamma\in\CH^n(Y\times X)$ between smooth projective complex varieties of dimension $n$ preserves the classes of coniveau $\geq 1$, and the classes of strong coniveau $\geq 1$. The assertion concerning coniveau follows from the formula $N^1H^l(X,A)=\ker\left(H^l(X,A)\to H^0(X,\mathcal{H}^l(A))\right)$ recalled in \S\ref{coniveautorsion} below, and from the fact that the correspondence $\Gamma$ induces a morphism $\Gamma_*:H^0(Y,\mathcal{H}^l(A))\to H^0(X,\mathcal{H}^l(A))$ (see \cite[Proposition A.1]{colliotvoisin}). The assertion for strong coniveau follows from the equality (\ref{alternastrong}).
In view of these facts and of Hironaka's theorem on resolution of singularities, we may apply \cite[Lemma~1.9]{voisinnotes} with $I(X)=N^1H^l(X,A)/\widetilde N^1H^l(X,A)$. This lemma reduces us to proving the invariance of $N^1H^l(X,A)/\widetilde N^1H^l(X,A)$ under a blow-up 
$\pi: Y\to X$ of a smooth projective complex variety in a smooth center. (Alternatively, we could have used the weak factorization theorem \cite[Theorem 0.1.1]{weakfacto}.)

Computing the cohomology of a blow-up shows that $H^l(Y,A)$ is generated by the image of the injective morphism $\pi^*:H^l(X,A)\to H^l(Y,A)$ and by classes supported on the exceptional divisor $E$ of $\pi$. The latter classes have strong coniveau $\geq~1$ since $E$ is smooth. Moreover, since $\alpha=\pi_*\pi^*\alpha$, the functoriality of coniveau $\geq 1$ and strong coniveau $\geq 1$ classes under the action of correspondences shows that $\alpha$ has coniveau $\geq 1$ if and only if $\pi^*\alpha$ has coniveau $\geq 1$, and that $\alpha$ has strong coniveau $\geq 1$ if and only if $\pi^*\alpha$ has strong coniveau $\geq 1$. These facts imply the desired equality $N^1H^l(X,A)/\widetilde N^1H^l(X,A)=N^1H^l(Y,A)/\widetilde N^1H^l(Y,A)$.
\end{proof}

\begin{corollary}
\label{rational}
If $X$ is a smooth projective complex variety which is stably rational, then $\widetilde N^1H^l(X,A)=N^1H^l(X,A)$ for every $l\geq 0$ and every abelian group $A$.
\end{corollary}

Corollary \ref{rational} could also have been deduced from the following proposition.

\begin{proposition}
\label{decdiag}
If a smooth projective complex variety $X$ admits an integral cohomological decomposition of the diagonal, then $\widetilde N^1H^l(X,A)=N^1H^l(X,A)=H^l(X,A)$ for every $l\geq 1$ and every abelian group $A$.
\end{proposition}

\begin{proof}
Choose $\alpha\in H^l(X,A)$. Let $p,q:X\times X\to X$ be the two projections and let $[\Delta]=[x\times X+\Gamma]$ be the decomposition, where the support $Z\subset X\times X$ of $\Gamma$ satisfies $q(Z)\subsetneq X$.
Let $Y$ be a disjoint union of resolutions of singularities of the images by $q$ of the irreducible components of $Z$, with induced morphism $f:Y\to X$. Let $\widetilde{Z}\to Z$ be a resolution of singularities such that, letting
$\pi:\widetilde{Z}\to X\times X$ denote its composition with the inclusion, there exists a morphism $g:\widetilde{Z}\to Y$ with $q\circ\pi=f\circ g$. Let $\widetilde{\Gamma}$ be a cycle on $\widetilde{Z}$ such that $\pi_*\widetilde\Gamma=\Gamma$. Then 
$$\alpha=q_*([\Delta]\smile p^*\alpha)=q_*([\Gamma]\smile p^*\alpha)=q_*\pi_*([\widetilde\Gamma]\smile \pi^*p^*\alpha)=f_*g_*([\widetilde\Gamma]\smile \pi^*p^*\alpha)$$
is in the image of $f_*$, hence has strong coniveau $\geq 1$. This proves the proposition.
\end{proof}

Propositions \ref{birinv} and \ref{decdiag} do not hold in general for higher coniveau; we will see later that there are even rational varieties where $\widetilde N^2H^l(X,\ZZ)\neq N^2 H^l(X,\ZZ)$ (see Theorem \ref{main} (ii)). 

Beyond the above general results, coniveau and strong coniveau may be shown to coincide in particular geometric situations, as in the next example.

\begin{example}
\label{exAJ}
Let $X$ be a smooth projective complex threefold such that there exist a smooth projective surface $F$ and a correspondence $\Gamma\in \CH^2(F\times X)$ for which $[\Gamma]_*:H^3(F,\ZZ)\to H^3(X,\ZZ)$ is surjective. Then 
$\widetilde N^1 H^l(X,\ZZ)=N^1 H^l(X,\ZZ)$ for all $l\ge 0$. For $l\neq 3$, this follows from Proposition \ref{egaliteconiveau}. For $l=3$, 
take a smooth ample divisor $i:C\hookrightarrow F$ so that $i_*:H^1(C,\ZZ)\to H^3(F,\ZZ)$ is surjective by the weak Lefschetz theorem, represent $(i,\Id)^*\Gamma\in \CH^2(C\times X)$ by a codimension $2$ cycle $Z$ on $C\times F$, and let $\pi:\widetilde{Z}\to C\times X$ be a resolution of singularities of the support of~$Z$. Then $(p_2\circ\pi)_*:H^1(\widetilde{Z},\ZZ)\to H^3(X,\ZZ)$ is surjective as wanted.

Taking $F$ to be an appropriate Fano variety parametrizing curves on $X$, and $\Gamma$ to be the class of the universal curve, this argument applies to all smooth cubic threefolds \cite[Theorem~11.19]{clemensgriffiths}, general quartic threefolds \cite[Proposition~1]{letizia}, 
%Letizia onbly shows isomorphism modulo torsion
general sextic double solids \cite[Theorem 3.3]{ceresaverra} 
%same here might be torsion in H^3(F,Z)
and general Gushel--Mukai threefolds \cite[Theorem~p.~84]{iliev}. In the last three examples, the argument works for all $X$ whose Fano variety $F$ is a smooth surface (or even a surface with isolated singularities as its hyperplane section $C$ may then be chosen to avoid its singular locus).

Similarly, if $X\subset \PP^5$ is a smooth cubic fourfold, then the variety of lines $F$ is a smooth fourfold and the Abel--Jacobi map $q_*p^* : H^6(F,\ZZ) \to H^4(X,\ZZ)$ is an isomorphism (as it is dual to the Beauville--Donagi isomorphism of \cite[Proposition~4]{beauvilledonagi}). Hence $\widetilde N^1H^4(X,\ZZ)=N^1H^4(X,\ZZ)$.
\end{example}

\subsection{Coniveau \texorpdfstring{$\geq 1$}{>=1} classes and torsion classes}
\label{coniveautorsion}

The classes of coniveau $\geq 1$ are of particular interest. Letting $\mathcal{H}^q(A)$ denote the sheaf associated with the Zariski presheaf $U\mapsto H^q(U,A)$ on $X$, Bloch and Ogus \cite[Corollary 6.3]{blochogus} have shown the existence of a spectral sequence $E_2^{pq}=H^p(X,\mathcal{H}^q(A)) \Rightarrow H^{p+q}(X,A)$ converging to the coniveau filtration on $H^{p+q}(X,A)$. In particular, the kernel of the natural map $H^l(X,A)\to H^0(X,\mathcal{H}^{l}(A))$ consists of the classes of coniveau $\ge 1$.

The following proposition, a consequence of the Bloch--Kato conjecture proven by Voevodsky and Rost, had been conjectured by Bloch \cite[end of Lecture~5]{blochbook}. A proof may be found in
\cite[Proof of Theorem 1 (ii)]{blochsrinivas} for $l=3$, and 
in \cite[Th\'eor\`eme~3.1]{colliotvoisin} in general.

\begin{proposition}\label{torsionc}
If $X$ is a smooth complex variety, any torsion class $\alpha\in H^l(X,\ZZ)$ has coniveau $\geq 1$.
\end{proposition}

\begin{proof}
The image of $\alpha$ by the natural morphism $H^l(X,\ZZ)\to H^0(X,\mathcal{H}^l(\ZZ))$ is zero because $\mathcal{H}^l(\ZZ)$ is torsion-free by
\cite[Th\'eor\`eme 3.1]{colliotvoisin}. This concludes, since the kernel consists of classes of coniveau $\geq 1$. 
\end{proof}

\section{Topological obstructions}
\label{topo}

In this section, we describe two obstructions to integral cohomology classes of smooth projective complex varieties having high strong coniveau (Propositions \ref{obst2} and \ref{obst1}), which rely respectively on Steenrod operations (studied in \S\ref{Steenrod}) and on complex cobordism (considered in \S\ref{MU}).

 Our obstructions are of topological nature, reminiscent of Thom's counterexamples to the integral Steenrod problem \cite[Th\'eor\`emes III.5, III.9]{thom}. We formulate them in their natural generality (Propositions \ref{push2} and \ref{push1}, and \S\ref{MUoriented}).

\subsection{Steenrod operations}
\label{Steenrod}

The obstruction described in Proposition \ref{obst2} is based on carefully chosen elements $(S_j)_{j\geq 1}$ of the Steenrod algebra, which behave particularly well with respect to pushforward morphisms (see Proposition~\ref{SjGysin}).

Let $\A$ be the mod $2$ Steenrod algebra (see \cite{steenrod}). We recall that it is a graded $\ZZ/2$-algebra generated by degree $i$ elements $\Sq^i$ for $i\geq 0$, subject to the Adem relations
\begin{equation}
\label{Ádem}
\Sq^i\Sq^j=\sum_{k=0}^{\floor{j/2}}\binom{j-k-1}{i-2k}\Sq^{i+j-k}\Sq^k.
\end{equation}
The algebra $\A$ acts functorially on the mod $2$ cohomology of any topological space~$X$. For $\alpha,\beta\in H^*(X,\ZZ/2)$, this action satisfies Cartan's formula
\begin{equation}
\label{Cartan}
\Sq^i(\alpha\smile\beta)=\sum_{j=0}^i\Sq^j\alpha\smile\Sq^{i-j}\beta.
\end{equation}

Let $\A \Sq^1$ be the left ideal of $\A$ generated by~$\Sq^1$. 
 For $j\ge 1$, we define
\begin{equation}
\label{defSj}
S_j:=\Sq^{2^j-1} \cdots \Sq^7\Sq^3,
\end{equation}
which is an element of degree $2^{j+1}-j-3$ in $\A$ (by convention, $S_1$ is the unit of $\A$).

\begin{lemma}
\label{star1}
One has $\Sq^{2i-1}S_j\in \A \Sq^1$ for $j\geq 1$ and $1\leq i\leq 2^j-1$.
\end{lemma}

\begin{proof}
The proof is by induction on $j$. The statement is clear for $j=1$. For $j>1$, use the Adem relation (\ref{Ádem}) to write
$$\Sq^{2i-1}S_j=\Sq^{2i-1}\Sq^{2^j-1}S_{j-1}= \sum_{k=0}^{2^{j-1}-1} \binom{2^j-k-2}{2i-2k-1} \Sq^{2^j+2i-k-2}\Sq^k S_{j-1}.$$ 
As $2i-2k-1$ is odd, $\binom{2^j-k-2}{2i-2k-1}$ is even whenever $2^j-k-2$ is even. It follows that the
only terms that contribute are those with $k$ odd. Since $\Sq^k S_{j-1}\in \A \Sq^1$ for those $k$
by the induction hypothesis, the lemma is proved.
\end{proof}

Proposition \ref{WuAH} is a relative variant of Wu's theorem \cite[Theorem 11.4]{MS} which was proven by Atiyah and Hirzebruch \cite[Satz~3.2]{AH}. 

We denote by $\Sq=\Sq^0+\Sq^1+\dots$ the total Steenrod operation and by $\Sq^{-1}$ its inverse. If $E-E'$ is a virtual real vector bundle, we let $w(E-E')=w(E)w(E')^{-1}$ be its total Stiefel--Whitney class.

\begin{proposition}
\label{WuAH}
Let $f:Y\to X$ be a proper $\ci$-map between $\ci$-manifolds with virtual normal bundle $N_f:=f^*T_{X}-T_{Y}$. For all $\beta\in H^*(Y,\ZZ/2)$, one has
$$\Sq (f_*\beta)=f_*(\Sq(\beta)\smile w(N_f))$$ in $H^*(X,\ZZ/2).$
\end{proposition}

\begin{proof}
When $X$ and $Y$ are compact, this is exactly \cite[Satz 3.2]{AH} applied with $\lambda=\Sq^{-1}$ (in \emph{loc.~cit.}, $\Wu(\Sq,X)=\Sq^{-1}w(T_X)$ by Thom's definition of the Stiefel--Whitney classes \cite[p.~91]{MS}).
As noted in \cite[Proof of Proposition 1.22]{BWI}, the standing assumption that manifolds are compact in \cite[\S 3]{AH} is superfluous. Indeed, one may assume that $Y$ is connected and choose an injective immersion $i:Y\to \SS^m$ thanks to Whitney's theorem \cite[Chapter 2, Theorem 2.14]{Hirsch}. The proof of \cite[Satz 3.2]{AH} then goes through using the embedding $(f,i):Y\to X\times\SS^m$.
\end{proof}

We now apply the relative Wu theorem to the cohomology operation $S_j$.
 
Let $E$ and $E'$ be two real vector bundles of constant rank on a $\ci$-manifold $X$.
%The good definition of a virtual vector bundle would be a map X->BO and of a stably complex structure would be a lift to X->BU.
%For a finite-dimensional CW complex X such as a C^infty manifold, one always lands in a fixed skeleton of BO or BU hence in a fixed BO(n) or BU(n). Hence the definitions below.
A \textit{stably complex structure} on the virtual bundle $E-E'$ is a homotopy class of isomorphisms $\iota:E\oplus \RR^k\simeq E'\oplus F$ where $F$ is a complex vector bundle, and where one identifies $\iota$ with $(\iota,\Id):E\oplus \RR^{k+2}\simeq E'\oplus F\oplus\CC$.
%The definition of "complex oriented" given by Voisin in [Chow Rings, Decomposition of the Diagonal, and the Topology of Families, §6.1] seems incorrect as a stably complex structure on a vector bundle in her sense would not specify an orientation of the vector bundle. It is more a definition of "complex orientable".
 A \textit{complex orientation} of a $\ci$-map $f:Y\to X$ between $\ci$-manifolds is a stably complex structure on the virtual normal bundle $N_f:=f^*T_{X}-T_Y$ of $f$.

\begin{proposition}
\label{SjGysin}
Let $f:Y\to X$ be a complex oriented proper $\ci$-map between $\ci$-manifolds. For all $j\geq 1$ and all $\beta\in H^*(Y,\ZZ/2)$ such that $\Sq^1(\beta)=0$, one has
$$S_j (f_*\beta)=f_*S_j(\beta)$$ in $H^*(X,\ZZ/2).$
\end{proposition}

\begin{proof}
The odd degree Stiefel--Whitney classes of a complex vector bundle vanish by \cite[Problem 14-B]{MS}. In view of Whitney's sum formula, the same holds for the odd degree Stiefel--Whitney classes of a stably complex virtual vector bundle such as $N_f=f^*T_{X}-T_{Y}$.

We prove the proposition by induction on $j$. The statement is clear for $j=1$. So assume $j>1$. We may assume that $\beta\in H^k(Y,\ZZ/2)$. By the induction hypothesis, 
$$S_j (f_*\beta)=\Sq^{2^j-1}S_{j-1}(f_*\beta)=\Sq^{2^j-1}(f_*S_{j-1}(\beta)).$$
By Proposition \ref{WuAH}, the class $S_j (f_*\beta)$ is the image by $f_*$ of the component of degree $2^{j+1}-j-3+k$ of $\Sq S_{j-1}(\beta)\smile w(N_f)$. Lemma \ref{star1}, the hypothesis that $\Sq^1(\beta)=0$ and the fact that $w(N_f)$ has no odd degree component show at once that 
\phantom\qedhere
$$S_j (f_*\beta)=f_*(\Sq^{2^j-1}S_{j-1}(\beta)\smile w_0(N_f))=f_*S_j(\beta).\eqno\qed$$
\end{proof}

We may now state the two main results of this section.

\begin{proposition}
\label{push2}
Let $f:Y\to X$ be a complex oriented proper $\ci$-map between $\ci$-manifolds, and let $\beta\in H^k(Y,\ZZ/2)$ be such that $\Sq^1(\beta)=0$. If $j\geq k$ and $j\geq 2$, then $S_j(f_*\beta)=0$ in $H^*(X,\ZZ/2)$.
\end{proposition}

\begin{proof}
By Proposition \ref{SjGysin}, one has $S_j (f_*\beta)=f_*S_j(\beta)=f_*\Sq^{2^j-1}S_{j-1}(\beta)$. Since the class $S_{j-1}(\beta)$ has degree $2^{j}-j-2+k<2^j-1$, one has $\Sq^{2^j-1}S_{j-1}(\beta)=0$, and it follows that $S_j(f_*\beta)=0$.
\end{proof}

Since a morphism of smooth complex varieties is canonically complex oriented, we deduce at once the following proposition.

\begin{proposition}
\label{obst2}
On a smooth complex variety $X$, choose $\alpha\in H^l(X,\ZZ)$; let $\overline{\alpha}\in H^l(X,\ZZ/2)$ denote the reduction modulo $2$ of $\alpha$, and let $c$ and $j$ be integers such that $l\leq 2c+j$ and $j\geq 2$. If $S_j(\overline{\alpha})\neq 0$, then $\alpha$ has strong coniveau $<c$.
\end{proposition}

\subsection{Complex cobordism}
\label{MU}

We now use complex cobordism to obtain refinements of Proposition \ref{push2} when $k\leq 2$ and of Proposition \ref{obst2} when $l\leq 2c+2$. These improvements are not needed in the proofs of our main theorems.

To every topological space $X$, one can associate its complex cobordism ring $MU^*(X)$, which is a graded ring. These rings form a generalized cohomology theory, represented by the complex cobordism spectrum $\mathbf{MU}$ (see for instance \cite[Chapter 12]{switzer} or \cite{adams-stable}). In this article, we will be only interested in the complex cobordism of $\ci$-manifolds. In this setting, Quillen \cite[\S 1]{quillen} gave  
%see [https://www.math.ucla.edu/~carrick/Thesis.pdf, \S 3] for more details. Also \cite{karoubi}.
a concrete description of $MU^*(X)$ which we briefly recall.

 Let $X$ be a $\ci$-manifold. Two proper $\ci$-maps $g_0:Z_0\to X$ and $g_1:Z_1\to X$ that are complex oriented (in the sense recalled in \S\ref{Steenrod}) are said to be \textit{cobordant} if there exists a complex oriented proper $\ci$-map $\widetilde{g}:\widetilde{Z}\to X\times\RR$ such that, for $i\in\{0,1\}$,  $\widetilde{g}$ is transversal to the inclusion $X\times\{i\}\hookrightarrow X\times\RR$, and $g_i$ identifies with $\widetilde{g}|_{\widetilde{g}^{-1}(X\times\{i\})}$ as a complex oriented $\ci$-map. For $r\in\ZZ$, Quillen identifies $MU^r(X)$ with the set of cobordism classes $[g]$ of complex oriented proper $\ci$-maps $g:Z\to X$ from a $\ci$-manifold $Z$ of dimension $\dim(X)-r$, with disjoint union as a group law \cite[Proposition 1.2]{quillen}. 
 %It really makes no difference to consider C^0 or C^infty maps of manifolds (by C^infty approximation). Quillen chooses that everything is C^infty and I follow him.

The above definition makes it clear how to construct Gysin morphisms in complex cobordism. If $f:Y\to X$ is a proper complex oriented map between $\ci$-manifolds, one can define $f_*:MU^r(Y)\to MU^{r+\dim(X)-\dim(Y)}(X)$ by sending the class $[g]$ represented by a complex oriented proper $\ci$-map $g:Z\to Y$ to $[f\circ g]$ \cite[\S 1.4]{quillen}.

As complex cobordism is the universal complex oriented cohomology theory \cite[II, Lemma 4.6]{adams-stable},
%\cite[Proposition 4.4.6]{Kochman}, 
the complex orientation of cohomology with integral coefficients \cite[II, Example (2.2)]{adams-stable} 
%\cite[\S 4.3 Example (1)]{Kochman}
yields a natural transformation \cite[II, Example~(4.7)]{adams-stable} 
\begin{equation}
\label{MUHZ}
\mu:MU^*(-)\to H^*(-,\ZZ).
\end{equation}
When $X$ is a $\ci$-manifold, the image by $\mu$ of a class in $MU^*(X)$ represented by a complex oriented proper $\ci$-map $g:Z\to X$ is $\mu([g])=g_*1\in H^*(X,\ZZ)$, where $g_*:H^*(Z,\ZZ)\to H^*(X,\ZZ)$ is the Gysin morphism (the complex orientation of~$g$ induces an orientation of the virtual vector bundle $N_g$, hence allows to define Gysin morphisms by \cite[V, Definition~2.11~(b)]{rudyak}) and $1\in H^0(Z,\ZZ)$ is the unit   (see \cite[\S 9]{karoubi}). It is clear from this description that if $f:Y\to X$ is a proper complex oriented map between $\ci$-manifolds and if $\gamma\in MU^*(Y,\ZZ)$, then 
\begin{equation}
\label{MUHGysin}
\mu(f_*\gamma)=f_*\mu(\gamma).
\end{equation}

The next proposition is well-known
(see for instance \cite[p.~468]{totaro}).
%or \cite[p.~963]{quick}, but beware that the proof sketched in \cite{quick} is incorrect).
%and may easily be deduced from Milnor's paper on complex cobordism \cite[Theorem 2]{Milnor}.
In the setting of oriented cobordism, the last assertion originates from Thom's work \cite[Th\'eor\`eme II.20]{thom}.
We let $\A\Sq^1\hspace{-.3em}\A$ be the two-sided ideal of $\A$ generated by~$\Sq^1$. 

\begin{proposition}
\label{SqMU}
If $X$ is a topological space and $r\geq 0$, the image of the morphism
\begin{equation}
\label{MUHZ/2}
MU^r(X)\xrightarrow{\mu} H^r(X,\ZZ)
\end{equation}
induced by (\ref{MUHZ}) is killed by stable integral cohomological operations of positive degree, and the reduction modulo $2$ of a class in the image of (\ref{MUHZ/2}) is annihilated by $\A\Sq^1\hspace{-.3em}\A$.
\end{proposition}

\begin{proof}
Consider a stable integral cohomological operation of degree $k>0$, induced by a map of spectra $\nu: \mathbf{H}\ZZ\to\Sigma^k\mathbf{H}\ZZ$, where $\mathbf{H}\ZZ$ is the Eilenberg--MacLane spectrum representing cohomology with integral coefficients, and let $\mu:\mathbf{MU}\to\mathbf{H}\ZZ$ be the map of spectra inducing (\ref{MUHZ}). The morphism $H^k(\mathbf{H}\ZZ,\ZZ)\to H^k(\mathbf{MU},\ZZ)$ induced by $\mu$ sends the class represented by $\nu$ to that represented by $\nu\circ\mu$. Since $H^k(\mathbf{H}\ZZ,\ZZ)$ is torsion by \cite[\S 6]{cartan} and $H^k(\mathbf{MU},\ZZ)$ is torsion-free by \cite[I, \S 3]{adams-stable}, we deduce that $\nu\circ \mu$ is homotopically trivial, which proves the first assertion.

Let $\rho$ and $\beta_{\ZZ}$ denote reduction modulo $2$ and the integral Bockstein. The second assertion follows from the first since the stable integral cohomological operation $\beta_{\ZZ}\circ a \circ \rho$ is a lift of $\Sq^1\circ\ a\circ\rho$ for all $a\in \A$.
\end{proof}

We finally reach the goal of this section.

\begin{proposition}
\label{push1}
Let $f:Y\to X$ be a complex oriented proper $\ci$-map between $\ci$-manifolds, and let $\beta\in H^k(Y,\ZZ)$. If $k\leq 2$, then $f_*\beta$ is in the image of the morphism $\mu: MU^*(X)\to H^*(X,\ZZ)$ induced by (\ref{MUHZ}) and its reduction modulo $2$ is killed by $\A\Sq^1\A$.
\end{proposition}

\begin{proof}
By the easy half of \cite[Theorem 2.2]{totaro}, whose proof is valid for any $X$, there exists a class $\gamma\in MU^*(Y)$ with $\mu(\gamma)=\beta$. By (\ref{MUHGysin}), one has $f_*\beta=\mu(f_*\gamma)$, which
proves the first assertion. The second now follows from Proposition \ref{SqMU}.
\end{proof}

As a morphism of smooth complex varieties is canonically complex oriented, we deduce:

\begin{proposition}
\label{obst1}
Let $X$ be a smooth complex variety, choose $\alpha\in H^l(X,\ZZ)$, let $\overline{\alpha}\in H^l(X,\ZZ/2)$ be the reduction modulo $2$ of $\alpha$, and let $c$ be such that $l\leq 2c+2$. If $\alpha$ is not in the image of $\mu: MU^*(X)\to H^*(X,\ZZ)$, or if $\overline{\alpha}$ is not killed by $\A\Sq^1\hspace{-.3em}\A$, then $\alpha$ has strong coniveau $<c$.
\end{proposition}

\subsection{A remark on the topological obstructions}
\label{MUoriented}

In the statements of Propositions \ref{push2} and \ref{push1}, one could replace the hypothesis that $f$ is complex oriented by the weaker hypothesis that its virtual normal bundle $N_f:=f^*T_{X}-T_{Y}$ is $MU$-oriented in the sense of \cite[\S V.1]{rudyak}.

Indeed, only two properties of a complex oriented map $f$ are used in the proofs of Propositions \ref{push2} and \ref{push1}: the existence of a Gysin morphism $f_*:MU^*(Y)\to MU^*(X)$ if $f$ is proper, and the fact that all the odd Stiefel--Whitney classes of $N_f$ vanish. Under the sole hypothesis that $N_f$ is $MU$-oriented, the first property is provided by \cite[V, Definition~2.11~(b)]{rudyak}. As for the second, every $MU$-oriented vector bundle $E$ has vanishing odd Stiefel--Whitney classes. To see it, write Thom's definition of Stiefel--Whitney classes based on Steenrod operations and on the Thom class of $E$ in mod $2$ cohomology \cite[p.~91]{MS}, notice that this Thom class lifts to complex cobordism as $E$ is $MU$-oriented, and apply Proposition \ref{SqMU}.

We refer to \cite{wilson} for an example of a real vector bundle which is $MU$-oriented but has no stably complex structure, which shows that this is a genuine generalization of Propositions \ref{push2} and \ref{push1}.
We will not use this generalization in what follows.

\section{Algebraic approximations of classifying spaces}
\label{BG}

The most direct way of producing $X$ and $\alpha$ where the obstructions of the previous section take place comes from algebraic
approximations to a classifying space $BG$. This was the construction used in the original counterexamples to the integral Hodge conjecture due to Atiyah and Hirzebruch \cite{AH2}.

\subsection{Torsion examples}
\label{2torsion}

We first use this technique with $G=(\ZZ/2)^s$ to prove Theorem \ref{mainthm}.

\begin{lemma}
\label{Sjnonzero}
For all $1\leq j\leq s$, there exists $\zeta\in H^s((\ZZ/2)^s,\ZZ/2)$ with $S_j\Sq^1(\zeta)\neq~0$.
\end{lemma}

\begin{proof}
The K\"unneth formula yields an algebra isomorphism
$$H^*((\ZZ/2)^s,\ZZ/2)=\ZZ/2[x_1,\ldots,x_s]$$
with generators $x_i$ in degree $1$. Take $\zeta=x_1\cdots x_s$. 
Combining \cite[Definition~2.4.9 and Proposition 5.8.4]{walkerwood} shows that $S_j\Sq^1(\zeta)\neq 0$. Alternatively, developing 
$$S_j\Sq^1(\zeta)=\Sq^{2^j-1} \cdots\Sq^3\Sq^1(x_1\cdots x_s)$$
using Cartan's formula, we get a polynomial in which the monomial $$x_1^{2^j}x_2^{2^{j-1}} \cdots x_j^2 x_{j+1}\cdots x_{s-1}x_s$$
appears non-trivially. It follows that $S_j\Sq^1(\zeta)\neq 0$.
\end{proof}

\begin{lemma}
\label{GodeauxSerre}
For all $1\leq j\leq s$, there exist a smooth projective complex variety~$V$ with torsion canonical bundle
and a class $\xi\in H^s(V,\ZZ/2)$ with $S_{j}\Sq^1(\xi)\neq 0$.
\end{lemma}

\begin{proof}
We use the Godeaux--Serre construction. Define $m:=2^{j+1}-j-1+s$ and let $\ZZ/2$ act on $\PP^{2m+1}$ by the involution $$\iota:(X_0,\dots,X_{m},X_{m+1},\dots,X_{2m+1})\mapsto (X_0,\dots,X_{m},-X_{m+1},\dots,-X_{2m+1}).$$
The fixed locus of this action has dimension $m$. Let $Z\subset\PP^{2m+1}$ be a general complete intersection of $m+1$ $\iota$-invariant quadrics. The smooth projective variety $Z$ has trivial canonical bundle. Since $\iota$ acts freely on $Z$, the quotient $Y:=Z/\iota$ is a smooth projective variety with torsion canonical bundle (when $m=2$, this is the classical construction of Enriques surfaces). We choose $V:=Y^{s}$. By \cite[Proposition~6.6 and its proof]{AH2}, there exist maps $a:Y\to B\ZZ/2$ and ${b:Y\to \PP^{\infty}(\CC)}$ such that $(a,b):Y\to B\ZZ/2\times \PP^{\infty}(\CC)$ is an $(m-1)$-homotopy equivalence. Consequently, $a^*:H^*(\ZZ/2,\ZZ/2)\to H^*(Y,\ZZ/2)$ is injective in degree $\leq m-1$. By the K\"unneth formula, $(a^{s})^*:H^*((\ZZ/2)^{s},\ZZ/2)\to H^*(V,\ZZ/2)$ is also injective in degree $\leq m-1$.
Applying Lemma \ref{Sjnonzero} yields a class $\zeta\in H^{s}((\ZZ/2)^{s},\ZZ/2)$ such that $S_j\Sq^1(\zeta)\neq 0$.
Setting $\xi:=(a^{s})^*\zeta$, one has $S_j\Sq^1(\xi)=(a^{s})^*S_j\Sq^1(\zeta)\neq 0$ by our choice of $m$.
\end{proof}

Now comes the proof of Theorem \ref{mainthm}. The crucial case is the $c=1$ case, where one can use that torsion classes always have coniveau $\geq 1$. The statement for higher values of $c$ follows using product and blow-up constructions.

\begin{theorem}
\label{main}
For all $c\geq 1$ and $l\geq 2c+1$, there exists a smooth projective complex variety $X$ such that the inclusion $\widetilde N^cH^l(X,\ZZ)\subset N^cH^l(X,\ZZ)$ is strict. Moreover,
\begin{enumerate}[(i)]
\item one can choose $X$ with torsion canonical bundle.
\item if $c\ge 2$ and $l\ge 2c+1$, one can choose $X$ to be rational.
\end{enumerate}
\end{theorem}

\begin{proof}
Let $V$ and $\xi$ be as in Lemma \ref{GodeauxSerre} applied with $j=s=l-2c+1$.
%Remark : can also choose s=l-2c+1 and j=l-2c if l>=2c+2.
Let $T$ be a smooth projective complex variety of dimension $c-1$, which we choose to have torsion  canonical bundle if we want to ensure (i). Let $\lambda\in H^{2c-2}(T,\ZZ)$ be the class of a point $t\in T$. We define $X:= V\times T$ with projections $p:X\to V$ and $q:X\to T$, and we set $\alpha:=p^*\beta_{\ZZ}(\xi)\smile q^*\lambda$, where $\beta_{\ZZ}$ is the integral Bockstein.

Let $\overline{\alpha}$ and $\overline{\lambda}$ be the reductions modulo $2$ of $\alpha$ and $\lambda$. 
Since $\overline{\lambda}$ is killed by all positive degree elements of $\A$ for degree reasons and since the reduction modulo~$2$ of $\beta_{\ZZ}(\xi)$ is $\Sq^1(\xi)$, Cartan's formula (\ref{Cartan}) shows at once that 
$S_j(\overline{\alpha})=p^*S_j\Sq^1\xi\smile q^*\overline{\lambda}$.
This class being nonzero, Proposition \ref{obst2} implies that $\alpha$ has strong coniveau $<c$.

The class $\alpha$ is the pushforward of $\beta_{\ZZ}(\xi)$ by the codimension $c-1$ closed immersion $V\times \{t\}\to V\times T$. As $\beta_{\ZZ}(\xi)$ is torsion, it has coniveau $\geq 1$ by Proposition~\ref{torsionc}, and it follows that $\alpha$ has coniveau $\geq c$. This finishes the proof of (i).

For (ii),  we let $W$ be as in part (i), admitting a class in $\sigma \in N^{c-1}H^{l-2}(W,\ZZ)$ so that $S_{j}(\bar \sigma)\neq 0$ for $j=l-2c+1$. Let $n=\dim W$. Let $W\to \PP^{n+2}$ be the composition of a projective embedding of $W$ and a generic projection to $\PP^{n+2}$. Performing an embedded resolution of the image $W_0$ of $W$ in $\PP^{n+2}$, we find a smooth rational variety $Y$ of dimension $n+2$, which contains a smooth subvariety $\widetilde W$, which admits a birational morphism $\widetilde W\to W$. By construction, $\widetilde W$ then also carries a class $\gamma\in N^{c-1} H^{l-2}(\widetilde W,\ZZ)$ for which $S_j(\bar \gamma)\neq 0$. Now let $X$ be the blow-up of $Y$ along $\widetilde W$ with exceptional divisor $E$ and take the class $\alpha\in H^{l}(X,\ZZ)$ to be $i_*\pi^*\gamma$, where $i:E\to X$ is the inclusion, and $\pi: E\to \widetilde W$ is the projective bundle. Then by Lemma \ref{SjGysin}, $S_j(\bar{\alpha})=i_*\pi^*(S_j\bar{\gamma})\neq 0$, and we conclude that $\alpha$ has strong coniveau $<c$ by Proposition \ref{obst2}. On the other hand $\alpha$ is the pushforward of a coniveau $\geq c-1$ class from a codimension 1 closed immersion $i:E\to X$, hence it has coniveau $\ge c$. This completes the proof.
\end{proof}

\begin{remark}
One does not need to appeal to Proposition~\ref{torsionc}, hence to the Bloch--Kato conjecture, to prove that $\beta_{\ZZ}(\xi)$  has coniveau $\geq 1$ in the proof of Theorem~\ref{main}. Indeed, it follows from the construction of $\xi\in H^s(V,\ZZ/2)$ given in Lemmas \ref{Sjnonzero} and \ref{GodeauxSerre} that $\xi=x_1\cdots x_s$ for some $x_i\in H^1(V,\ZZ/2)$. Since $\beta_{\ZZ}(x_i)\in H^2(V,\ZZ)$ is $2$-torsion, the Lefschetz $(1,1)$ theorem shows that it is an algebraic class, hence has coniveau~$\geq 1$. It follows that $x_i$ lifts to an integral class in restriction to a dense open subset of $V$. Hence so does $\xi=x_1\cdots x_s$, showing that $\beta_{\ZZ}(\xi)$ has coniveau $\geq 1$.
\end{remark}

\begin{remark}
It was pointed out to us by Yagita that one can obtain $p$-torsion examples for all odd prime numbers $p$, using analogous arguments which we now briefly sketch. Let $\A(p)$ be the mod $p$ Steenrod algebra, and let $Q_1\in\A(p)$ be the element introduced by Milnor in \cite[\S 6]{milnorsteenrod}. There exists a $p$-torsion class in $H^3((\ZZ/p)^2,\ZZ)$ whose reduction modulo $p$ is not annihilated by $Q_1$ (for instance, the element $\beta_{\ZZ}(u_1u_2)$ in the notation of \cite[Proof of (6.7)]{AH2}). As in \cite[Proof of (6.6)]{AH2}, we deduce the existence of a smooth projective complex variety $X$ and of a $p$-torsion class $\alpha\in H^3(X,\ZZ)$ whose reduction modulo~$p$ is not killed by $Q_1$. The class $\alpha$ has coniveau~$\geq 1$ by Proposition~\ref{torsionc}. It cannot have strong coniveau $\geq 1$ because the cohomological operation $Q_1$ commutes with pushforwards by $\ci$-maps of oriented compact $\ci$-manifolds 
%in fact oriented \ci maps of \ci manifolds
(use \cite[Satz 2.12 and Satz 3.2]{AH} and the fact that the mod $p$ Bockstein commutes with such pushforwards),
%computing the mod p cohomology of MSO
yet vanishes on reductions modulo $p$ of degree~$1$ integral cohomology classes (see \cite[\S VI.1]{steenrod}).
\end{remark}

\subsection{Torsion-free examples} 

Since the cohomology of a finite group is torsion in positive degree, the examples of integral cohomology classes for which coniveau and strong coniveau differ that can be obtained using classifying spaces of finite groups live in cohomology groups that have torsion. To produce torsion-free examples, we resort to classifying spaces of linear algebraic groups, namely of the exceptional group $G_2$, as in \cite{pirutkayagita}.

\begin{proposition}
\label{G2}
There exists a smooth projective complex variety $X$ 
such that $H^4(X,\ZZ)$ is torsion-free and the inclusion $\widetilde N^1H^4(X,\ZZ)\subset N^1H^4(X,\ZZ)$ is strict.
\end{proposition}

\begin{proof}
It follows from the work of Borel (notably \cite[Proposition 19.2]{borel1}
%in this reference, the second occurence of $H^*(G,K_2) should be replaced with H^*(BG, K_2) (K_2 field of characteristic 2).
and \cite[Th\'eor\`eme 17.3 (c)]{borel2}) that $H^4(BG_2,\ZZ)$ is torsion-free and contains a class whose reduction modulo $2$ is not killed by $\Sq^3$ (see \cite[\S 2.4, Theorem 2.19]{akbulut} for a proof of this precise statement). The same property holds for the classifying space $B(G_2\times \mathbb{G}_m)=BG_2\times B\mathbb{G}_m=BG_2\times\PP^{\infty}(\CC)$. 

By Ekedahl's construction of algebraic approximations to classifying spaces of reductive groups \cite[Theorem 1.3]{ekedahl}, there exist a smooth projective complex variety $X$ and a map $a:X\to B(G_2\times \mathbb{G}_m)$ such that the pull-back morphism $a^*: H^*(B(G_2\times \mathbb{G}_m),\ZZ)\to H^*(X,\ZZ)$ is an isomorphism in degree $\leq 8$. It follows from the five lemma that $a^*: H^*(B(G_2\times \mathbb{G}_m),\ZZ/2)\to H^*(X,\ZZ/2)$ is an isomorphism in degree $\leq 7$. We deduce that $H^4(X,\ZZ)$ is torsion-free and that there exists a class $\alpha\in H^4(X,\ZZ)$ in the image of $a^*$ whose reduction modulo $2$ is not killed by $\Sq^3$.

The class $\alpha$ has strong coniveau $0$ by Proposition \ref{obst2} or by Proposition \ref{obst1}. Edidin and Graham \cite[Theorem 1 (c)]{edidingraham} 
%check published version
(see also \cite[Theorem 2.14]{totarobook}) have shown the surjectivity of the cycle class map $\CH^2(B(G_2\times \mathbb{G}_m))\otimes_{\ZZ}\QQ\to H^4(B(G_2\times \mathbb{G}_m),\QQ)$. It follows that a multiple of $\alpha$ is algebraic. As a consequence, $\alpha$ restricts to a torsion class on a dense open subset $U\subset X$, hence has coniveau $\geq 1$ by Proposition \ref{torsionc} applied to $U$. The proposition is now proven.
\end{proof}

\begin{remark}
The class $\alpha\in H^4(X,\ZZ)$ considered in the proof of Proposition \ref{G2} is Hodge as a multiple of it is algebraic, but it is not algebraic since it has strong coniveau~$0$. This counterexample to the integral Hodge conjecture in a torsion-free cohomology group is parallel to the counterexamples to the integral Tate conjecture described by Pirutka and Yagita \cite[Theorem 1.1]{pirutkayagita}.
%The construction of an algebraic approximation in loc. cit. seems however incomplete as it does not analyse correctly the stable locus. 
\end{remark}

\section{Low-dimensional examples}
\label{low}
\def\a{\alpha}
\def\b{\beta}

The examples of Section \ref{BG} are relatively simple and work for any coniveau $c\ge 1$ and any degree $l\geq 2c+1$. On the other hand, the resulting varieties have quite high dimension. We now construct examples of dimension as low as $4$ and show that their dimensions is the lowest possible that may be attained using purely topological arguments.

\subsection{Construction of the examples}

Our first goal is to prove Theorem \ref{mainthm2}.

\subsubsection{A special bielliptic surface}
\label{bielliptic}
Let $E_1=\CC/(\ZZ+\ZZ \tau)$ and $E_2=\CC/(\ZZ+\ZZ i)$ be two elliptic curves, the second having complex multiplication by $i$. The group $G=\ZZ/4$ acts freely on $E_1\times E_2$ by translation by a $4$-torsion point $(u\mapsto u+\frac14)$ on the first factor, and by multiplication by $i$, $(v\mapsto iv)$ on the second. Let $S=(E_1\times E_2)/G$ be the quotient.

The morphism $(u,v)\mapsto (u,(1+i)v)$ on $\CC\times \CC$ induces a morphism $f:S\to S$ which is
finite étale of degree 2. Let $\alpha \in H^1(S,\ZZ/2)$ be the corresponding class.
\begin{lemma}
\label{calculbielliptique}
There is a class $\beta\in H^1(S,\ZZ/2)$ such that $\a^3\b\neq 0$ and
$\b^2=0$.
\end{lemma}
\begin{proof}
There is a natural diffeomorphism $S\simeq\SS^1\times M$, where $M$ is the quotient of $\SS^1\times E_2$ by
the diagonal action of $\ZZ/4$, by translation by ${\frac14}$ on $\SS^1= \RR/\ZZ$,
and by multiplication by $i$ on $E_2$. Moreover, $\alpha$ is the pullback
by the second projection of the class (which we still denote by $\alpha \in H^1(M,\ZZ/2)$)
associated to the double cover $f:M\to M$ defined by
$(u,v)\mapsto(u,(1+i)v)$. Let $\b$ be the pullback to $S$ of the generator of
$H^1(\SS^1,\ZZ/2)$. It is clear that $\b^2=0$. To conclude, it suffices to show that
$\alpha^3\neq 0$ in $H^3(M,\ZZ/2)$; then also $\alpha^3\beta\neq 0$ by the K\"unneth theorem.

  Using the first projection, we may view $M$ as the total space of a fibration
$p:M\to \SS^1/G=\SS^1$ with fibers $E_2$ (and the monodromy on the fiber is
given by multiplication by $i$). We let $x$ and $y$ be the real coordinates on
the universal cover $\CC\simeq\RR^2$ of $E_2$. Let $H\subset M$ (resp. $K\subset M$) be the immersed $\ci$-hypersurface which intersects $p^{-1}(0)$ along $\{xy=0\}$ (resp.
$\{(x-1/2)(y-1/2)=0\}$) and is obtained by transporting the latter flatly in
all fibers of $p$ (note that $\{xy=0\}$ and $\{(x-1/2)(y-1/2)=0\}$) are invariant
by the monodromy). 

The immersed submanifolds $H$ and $K$ intersect transversally along a $1$-dimensional submanifold
$C\subset M$ which intersects $p^{-1}(0)$ along the two points $(0,1/2)$ and
$(1/2,0)$ and is obtained by transporting flatly these two points in all fibers of $p$ (it is a
circle in $M$ with degree 2 over the base of $p$). 

Let us introduce the following deformation $C'$ of $C$. Start with the point $(\epsilon ,1/2)$
in $p^{-1}(0)$ for some small $\epsilon>0$, and transport it flatly in the fibers of p. After going
twice around the base $\SS^1$ of $p$, one arrives at the point $(-\epsilon,1/2)$
of $p^{-1}(0)$, which can be connected by a very small arc to $(\epsilon,
1/2)$. The resulting loop $C'$ intersects $H$
transversally in one point. Letting $[H]$, $[K]$, $[C]$ and $[C']$ denote the mod $2$ cohomology classes of $H$, $K$, $C$ and $C'$ in $M$, we deduce that the intersection number $[C']\smile [H]=[C]\smile [H]=[K]\smile [H]^2$ is nonzero.

 The Leray spectral sequence for $p$ yields an exact sequence
\begin{equation}
\label{Leray}
0\to H^1(\SS^1,\ZZ/2)\to H^1(M,\ZZ/2)\to H^0(\SS^1,R^1p_*\ZZ/2)=\ZZ/2.
\end{equation}
 The classes $[H]$ and $[K]$ are nontrivial in restriction to the fibers of $p$, hence both project to the nonzero class in
$H^0(\SS^1,R^1p_*\ZZ/2)=\ZZ/2$. It follows from (\ref{Leray}) that we may write $[H]=[K]+p^*\omega$ for some $\omega\in H^1(\SS^1,\ZZ/2)$. We now compute
$[K]^3=[K]\smile([K]+p^*\omega)^2=[K]\smile [H]^2\neq 0$, since $\omega^2=0$.

We finally remark that $[K]=\a$. Indeed, the pullback of $K$ by the double
cover $f:M\to M$ is an immersed hypersurface in $M$ obtained by transporting flatly in the fibers of $p$ the boundary of the square with vertices $(1/2,0),
(0,1/2), (1/2,1), (1,1/2)$. It is clearly a boundary in $M$, as it bounds
the domain obtained by transporting flatly in the fibers of $p$ the
interior of the same square. Hence 
%$f:M\to M$ kills the class $[K]$ 
$f^*[K]$ vanishes
in $H^1(M,\ZZ/2)$. Since $[K]\neq 0$, this shows that $[K]=\a$. In particular, $\a^3=[K]^3\neq 0$.
\end{proof}

\subsubsection{A diagonal quotient construction}
\label{diagonalquotient}

Let $M$ be a connected $\ci$-manifold, and choose a nontrivial class $\varepsilon\in
H^1(M,\ZZ/2)$ with associated double cover $M'\to M$. We will consider the quotient $N$ of
$M'\times \SS^1$ by the diagonal action of $\ZZ/2$ (by the natural action on the left, by $-\Id$ on $\SS^1=\RR/\ZZ$ on the right).

   Using the first projection, we view $N$ as the total space of a
fibration $q:N\to M$ whose fibers are isomorphic to $\SS^1$. The two fixed
points of the action of $\ZZ/2$ on $\SS^1$ give rise to two sections of $q$ whose images are $\ci$-hypersurfaces of $N$ denoted by $D$ and~$D'$. Let $\delta:=[D]\in H^1(N,\ZZ/2)$ and $\delta':=[D']\in H^1(N,\ZZ/2)$ denote the cohomology classes of $D$ and $D'$. As $D$ and $D'$ do not meet, $\delta\smile \delta'=0$ in $H^2(N,\ZZ/2)$. On the one hand, $N\setminus (D\cup D')$ is connected (being the image of $M'\times (0,\frac12)\subset M'\times \RR/\ZZ$), so that $\delta+\delta'\neq 0$ in $H^1(N,\ZZ/2)$. On the other hand, the inverse image of $D\cup D'$
in $M'\times \SS^1$ is a boundary, showing that $\delta+\delta'$ is killed by the double cover $M'\times \SS^1\to N$.
It follows that $\delta+\delta'=q^*\varepsilon$ in $H^1(N,\ZZ/2)$, hence that
\begin{equation}
\label{cohoequation}
\delta^2=q^*\varepsilon \smile \delta
\end{equation}in $H^2(N,\ZZ/2)$.

\begin{lemma}
\label{cohoring}
The formula $\lambda+\mu\delta\mapsto q^*\lambda+q^*\mu\smile \delta$ induces a ring isomorphism 
$$H^*(M,\ZZ/2)[\delta]/(\delta^2-\varepsilon \delta)\isoto H^*(N,\ZZ/2).$$
\end{lemma}

\begin{proof}
This ring morphism is well-defined by (\ref{cohoequation}). 
To show that it is injective, choose $\lambda,\mu\in H^*(M,\ZZ/2)$ with  $q^*\lambda+q^*\mu\smile \delta=0$ and note that $\mu=q_*(q^*\lambda+q^*\mu\smile \delta)=0$ and
$\lambda+\mu\smile \varepsilon=q_*((q^*\lambda+q^*\mu\smile \delta)\smile \delta)=0$ by the projection formula. To show surjectivity, take $\alpha\in H^l(N,\ZZ/2)$. Then $q_*(\alpha-(q^*q_*\alpha)\smile\delta)=0$ by the projection formula, and the Leray spectral sequence for $q$ shows the existence of $\lambda\in H^l(M,\ZZ/2)$ such that 
$\alpha=q^*\lambda+(q^*q_*\alpha)\smile\delta$.
\end{proof}

\subsubsection{A fourfold}

Combining the constructions of \S\ref{bielliptic} and \S\ref{diagonalquotient}, we obtain a remarkable smooth projective fourfold.

\begin{proposition}
\label{fourfold}
There exist a smooth projective complex fourfold $Z$ and a $2$\nobreakdash-torsion class $\sigma\in H^3(Z,\ZZ)$ such that the reduction modulo $2$ of $\sigma^2$ is nonzero.
\end{proposition}

\begin{proof}
Let $E$ be an elliptic curve, and let $S$ and $\alpha,\beta\in H^1(S,\ZZ/2)$ be as in \S\ref{bielliptic}. Consider the double cover $S'\to S$ associated with $\alpha$, and let $Y$ be the smooth projective complex threefold obtained as the quotient of $S'\times E$ by the diagonal action of $\ZZ/2$ (by the natural action on the left, by $-\Id$ on $E\simeq(\SS^1)^2\simeq(\RR/\ZZ)^2$ on the right). 

Let $\pi:Y\to S$ be the morphism induced by the first projection, and let $Y'\to Y$ be the double cover associated with $\pi^*\beta$. We define $Z$ to be the smooth projective complex fourfold obtained as the quotient of $Y'\times E$ by the diagonal action of $\ZZ/2$ (by the natural action on the left, by $-\Id$ on $E\simeq(\SS^1)^2\simeq(\RR/\ZZ)^2$ on the right).

The variety $Z$ may be constructed from $S$ by applying four times the construction of \S\ref{diagonalquotient}. As a consequence, its cohomology ring with $\ZZ/2$ coefficients may be computed by four successive applications of Lemma \ref{cohoring}:
$$H^*(Z,\ZZ/2)=H^*(S,\ZZ/2)[\delta,\delta',\gamma,\gamma']/(\delta^2-\alpha\delta, \delta'^2-\alpha\delta', \gamma^2-\beta\gamma,\gamma'^2-\beta'\gamma').$$

Define $\sigma:=\beta_{\ZZ}(\gamma\delta)\in H^3(Z,\ZZ)$, where $\beta_{\ZZ}$ is the integral Bockstein. It is a $2$-torsion class.
Then the reduction of  $\sigma^2$ modulo 2 is equal to
$$(\Sq^1(\gamma\delta))^2=(\gamma^2\delta+\gamma\delta^2)^2=\gamma^4\delta^2+\gamma^2\delta^4=\gamma \beta^3\delta^2+\gamma\beta\delta\alpha^3=\alpha^3\beta\gamma\delta\neq 0,$$
where we used that $\beta^2=0$ and that $\alpha^3\beta\neq 0$ (see Lemma \ref{calculbielliptique}).
\end{proof}

We may now prove Theorem \ref{mainthm2}.

\begin{theorem}
\label{main2}
For $l\in\{3,4\}$, there is a smooth projective complex variety $X$ of dimension $l+1$
with torsion canonical bundle
such that the inclusion $\widetilde N^1H^l(X,\ZZ)\subset N^1H^l(X,\ZZ)$ is strict. 
\end{theorem}

\begin{proof}
Let $Z$ and $\sigma$ be as in Proposition \ref{fourfold}. If $l=3$, we define $X:=Z$ and $\alpha:=\sigma$. If $l=4$, we choose an elliptic curve $E$ and a class $\tau\in H^1(E,\ZZ)$ whose reduction modulo $2$ is nonzero, and we define $X:=Z\times E$ and $\alpha:=p_1^*\sigma\smile p_2^*\tau\in H^4(X,\ZZ)$.

In both cases, $\alpha$ is $2$-torsion, hence has coniveau $\geq 1$ by Proposition~\ref{torsionc}.

Let $\overline{\alpha}$, $\overline{\sigma}$ and $\overline{\tau}$ denote the reductions modulo $2$ of $\alpha$, $\sigma$ and $\tau$. If $l=3$, then $\Sq^3(\overline{\alpha})=\overline{\sigma}^2\neq 0$. If $l=4$, then $\Sq^3(\overline{\alpha})=p_1^*\Sq^3(\overline{\sigma})\smile p_2^*\overline{\tau}=p_1^*\overline{\sigma}^2\smile p_2^*\overline{\tau}\neq 0$ by Cartan's formula (\ref{Cartan}) since $\Sq^1(\overline{\sigma})=\Sq^1(\overline{\tau})=\Sq^3(\overline{\tau})=0$.
In both cases, Proposition \ref{obst2} applied with $j=2$ or Proposition \ref{obst1} show that $\alpha$ has strong coniveau~$0$.
\end{proof}

\begin{remark}
For $X$ as in the above theorem, any class in $H_k(X,\ZZ)$ is realizable as the class of a real submanifold of $X$ (see \cite[Corollaire II.28]{thom}). Thus the obstructions we use are really of `complex' nature. 
\end{remark}

\subsection{Optimality}
\label{optimality}

In this section, we prove Theorem \ref{main3}, thus showing that the examples of Theorem~\ref{main2} are optimal in the following sense: their dimensions are the lowest possible for which there are topological obstructions to the equality of coniveau and strong coniveau in cohomological degree $3$ and $4$ (see Remark \ref{rkconiveaus}).

\subsubsection{A vanishing result}
The following proposition will be used crucially in the proof of Theorem \ref{main3}.

\begin{proposition}
\label{vanishing}
Let $X$ be a compact complex fourfold and $\alpha\in H^4(X,\ZZ)$.
If $\overline{\alpha}$ denotes the reduction modulo $2$ of $\alpha$ and $\beta_{\ZZ}$ is the integral Bockstein, then
$$\beta_{\ZZ}\Sq^2(\overline{\alpha})=0$$ in $H^7(X,\ZZ).$
\end{proposition}

\begin{proof}
Let $\iota:\ZZ/2\to \QQ/\ZZ$ be the natural injection, and let $\partial$ denote the boundary maps associated with the short exact sequence $0\to\ZZ\to\QQ\to\QQ/\ZZ\to 0$. In view of the commutative exact diagram
\begin{equation*}
\begin{aligned}
\xymatrix
@R=0.3cm
{
0\ar[r]&\ZZ\ar^{2}[r]\ar^{\wr}[d]&
\ZZ\ar^{}[d] \ar[r]& \ZZ/2\ar[r]\ar^{\iota}[d]&0 \\
0\ar[r]&\ZZ\ar[r]&\QQ \ar[r]& \QQ/\ZZ\ar[r]&0
}
\end{aligned}
\end{equation*}
one has $\beta_{\ZZ}\Sq^2(\overline{\alpha})=\partial( \iota_*\Sq^2(\overline{\alpha}))$ in $H^7(X,\ZZ)$. It follows that
$$\beta_{\ZZ}\Sq^2(\overline{\alpha})\smile\beta=\partial( \iota_*\Sq^2(\overline{\alpha}))\smile\beta=\iota_*\Sq^2(\overline{\alpha})\smile\partial(\beta) \mbox{ in }H^8(X,\QQ/\ZZ)$$
for all $\beta\in H^1(X,\QQ/\ZZ)$, where the last equality follows from \cite[Lemma 2.6]{linking}.

Defining $\gamma$ to be the reduction modulo $2$ of $\partial(\beta)$,
we deduce that
\begin{equation}
\label{eq1}
\beta_{\ZZ}\Sq^2(\overline{\alpha})\smile\beta=\iota_*(\Sq^2(\overline{\alpha})\smile\gamma)\,\, \mbox{  in }H^8(X,\QQ/\ZZ).
\end{equation}Cartan's formula (\ref{Cartan}) and the vanishing of $\Sq^1(\overline{\alpha})$ imply that
\begin{equation}
\label{eq2}
\Sq^2(\overline{\alpha})\smile\gamma=\Sq^2(\overline{\alpha}\smile\gamma)+\overline{\alpha}\smile\gamma^2
\end{equation}in $H^8(X,\ZZ/2)$. Finally, letting $u_2(X)\in H^2(X,\ZZ/2)$ denote the second Wu class of $X$ defined in \cite[\S 11 p.~131-132]{MS}, we have
\begin{equation}
\label{eq3}
\Sq^2(\overline{\alpha}\smile\gamma)=\overline{\alpha}\smile\gamma\smile u_2(X)
\end{equation}in $H^8(X,\ZZ/2)$. Notice that the classes $\overline{\alpha}$, $\gamma$ and $u_2(X)$ are the reductions modulo $2$ of the integral cohomology classes $\alpha$, $\partial(\beta)$ and $c_1(X)$ (for the latter assertion, combine Wu's theorem \cite[Theorem 11.4]{MS} and \cite[Problem 14-B]{MS}). Since $\partial(\beta)$ is torsion and $H^8(X,\ZZ)=\ZZ$ has no torsion, we deduce that
$\overline{\alpha}\smile\gamma^2=\overline{\alpha}\smile\gamma\smile u_2(X)=0$. Combining equations (\ref{eq1}), (\ref{eq2}) and (\ref{eq3}) now shows that $\beta_{\ZZ}\Sq^2(\overline{\alpha})\smile\beta=0$. Since $\beta$ was arbitrary, Poincar\'e duality (see Proposition \ref{Poincare} below), implies the required vanishing $\beta_{\ZZ}\Sq^2(\overline{\alpha})=0$.
\end{proof}

For lack of an explicit reference to the literature, we include a proof of the following instance of Poincar\'e duality.

\begin{proposition}
\label{Poincare}
If $M$ is a compact oriented $\ci$-manifold of dimension $d$, the cup product pairings 
\begin{equation*}
\begin{alignedat}{4}
H^k(M,\ZZ/n)\times H^{d-k}(M,\ZZ/n)&\to H^d(M,\ZZ/n)=\ZZ/n\hspace{2em}\\
\textrm{and}\hspace{1em} H^k(M,\ZZ)\times H^{d-k}(M,\QQ/\ZZ)&\to H^d(M,\QQ/\ZZ)=\QQ/\ZZ\hspace{2em}
\end{alignedat}
\end{equation*}
are non-degenerate on both sides for all $k\geq 0$ and $n\geq 1$.
\end{proposition}

\begin{proof}
To prove the assertion with $\ZZ/n$ coefficients, run the proof of \cite[Proposition 3.38]{Hatcher} with $R=\ZZ/n$, noting that the morphism $h$ in \emph{loc.~cit.} is an isomorphism by the universal coefficient theorem \cite[Theorem 3.2]{Hatcher} and since $\ZZ/n$ is an injective $\ZZ/n$-module (this argument appears in \cite[\S 3.2.6]{SGA43}).

To prove that the second pairing is non-degenerate on the left, take a nonzero $\alpha\in H^k(M,\ZZ)$. Since $H^k(M,\ZZ)$ is finitely generated, there exists $n\geq 1$ such that $\alpha$ is not divisible by $n$, hence such that its image $\overline{\alpha}$ in $H^k(M,\ZZ/n)$ does not vanish. By the assertion with $\ZZ/n$ coefficients, we may find $\beta\in H^{d-k}(M,\ZZ/n)$ with $\overline{\alpha}\smile\beta\neq 0$. The cup product of $\alpha$ with the image of $\beta$ in $H^{d-k}(M,\QQ/\ZZ)$ is then nonzero.

To prove that the second pairing is non-degenerate on the right, take a nonzero class $\beta\in H^{d-k}(M,\QQ/\ZZ)$. It is the image of a class $\beta_{n}\in H^{d-k}(M,\ZZ/n)$ for some~$n$. Let $\beta_{mn}\in H^{d-k}(M,\ZZ/mn)$ be the class induced by  $\beta_{n}$ for $m\geq 1$.  For all $m\geq 1$, there exists a class $\alpha_{mn}\in H^k(M,\ZZ/mn)$ with $\alpha_{mn}\smile\beta_{mn}\neq 0$ by the assertion with $\ZZ/mn$ coefficients. Since the $H^k(M,\ZZ/mn)$ are finite, one may use Tychonoff's theorem to choose the $\alpha_{mn}$ compatible with each other. The image of $\alpha_{n}$ by the boundary map of $0\to\ZZ\xrightarrow{n}\ZZ\to\ZZ/n\ZZ\to 0$ is divisible by $m$ for all $m\geq 1$ as $\alpha_{n}$ lifts to $H^k(M,\ZZ/mn)$. This image vanishes since $H^{k+1}(M,\ZZ)$ is finitely generated, so that $\alpha_n$ lifts to a class $\alpha\in H^k(M,\ZZ)$. Since $\alpha\smile\beta\neq 0$, the proof is complete.
\end{proof}

\subsubsection{Lifting cohomology classes to complex cobordism}

By \cite[Theorem~2.2]{totaro}, the morphism $\mu:MU^l(X)\to H^l(X,\ZZ)$ induced by (\ref{MUHZ}) is surjective for all topological spaces $X$ and all $l\leq 2$. Proposition \ref{MUsurj} describes other cases where surjectivity holds.

\begin{proposition}
\label{MUsurj}
If $X$ is a compact complex manifold of dimension $n$, the map 
$$\mu:MU^l(X)\to H^l(X,\ZZ)$$
induced by (\ref{MUHZ}) is surjective if $l+3\geq 2n$ or if $(l,n)=(4,4)$.
\end{proposition}

\begin{proof}
The Atiyah--Hirzebruch spectral sequence 
$H^p(X,MU^{q}(pt))\Rightarrow MU^{p+q}(X)$
(for which apply \cite[Proposition 4.2.9]{Kochman}
% in adams sign missing at the top of p.215 as may be seen from Example p. 219
with $E=\mathbf{MU}$) and Milnor's computation of the cobordism ring of the point \cite[II, Theorem~8.1]{adams-stable} give an exact sequence
\begin{equation}
\label{AHshort}
MU^l(X)\xrightarrow{\mu} H^l(X,\ZZ)\xrightarrow{d_3} H^{l+3}(X,\ZZ)
\end{equation}
%The argument below appears here : https://mathoverflow.net/q/62644
The right-hand arrow $d_3$ of (\ref{AHshort}) makes sense for all finite-dimensional CW complexes~$X$ 
%This condition ensures the existence of the spectral sequence
and all~$l$, may be extended to all CW complexes by restriction to their $(l+4)$-skeleta, and the resulting cohomology operation commutes with suspension: it is a stable integral cohomology operation of degree $3$. It follows from \cite[Theorem~5.4~(b)]{kochman2} that there are exactly two such operations: the trivial one, and $\beta_{\ZZ}\Sq^2$ (as in \cite[Proposition~7.2]{AH2}, one may actually check that $d_3=\beta_{\ZZ}\Sq^2$). Both vanish on $H^l(X,\ZZ)$ (because $H^{l+3}(X,\ZZ)$ is torsion-free if $l+3\geq 2n$, and by Proposition~\ref{vanishing} if $(l,n)=(4,4)$). The proposition now follows from the exactness of~(\ref{AHshort}).
\end{proof}

\subsubsection{Vanishing of topological obstructions}

We finally reach the goal of \S\ref{optimality}.

\begin{theorem}
\label{main3}
Let $X$ be a compact complex manifold of dimension $n$ and let $\alpha\in H^l(X,\ZZ)$. If either $(l,n)=(3,3)$ or $(l,n)=(4,4)$, there exists a compact almost complex $\ci$-manifold $Y$ of complex dimension $n-1$, a $\ci$-map $f:Y\to X$ and a class $\beta\in H^{l-2}(Y,\ZZ)$ with $f_*\beta=\alpha$. 
\end{theorem}

\begin{proof}
In both cases, the map $MU^l(X)\to H^l(X,\ZZ)$ is surjective by Proposition~\ref{MUsurj}. This means that there exist a compact stably almost complex $\ci$-manifold $M$ of dimension $2n-l$ and a $\ci$-map $h: M\to X$ so that $h_*1=\alpha$ (see \S\ref{MU}).

Consider first the case $n=l=3$. In this case we take $Y=M\times \SS^1$, $f=h\circ pr_1 \colon Y\to X$ and $\beta=pr_2^*u$ where $u\in H^1(\SS^1,\ZZ)$ is the oriented generator. We claim that $Y$ admits an almost complex structure. Wu \cite{wu4} showed that an oriented real $4$-manifold $Y$ admits an almost complex structure if and only if there is an integral class $c\in H^2(Y,\ZZ)$ which lifts to the mod $2$ Stiefel--Whitney class of the tangent bundle $w_2(Y)\in H^2(Y,\ZZ/2)$, and such that $c^2=3\sigma(Y)+2\chi(Y)$, where $\sigma$ is the signature and $\chi$ is the Euler characteristic. In our case, we compute that $\sigma(Y)=\chi(Y)=w_2(Y)=0$ (to show that $w_2(Y)=0$, apply \cite[Problem 12-B]{MS} to the orientable $3$-manifold $M$), so we can simply take $c=0$.

The case for $n=l=4$ follows in a similar way, letting $Y=M\times \PP^1(\CC)$, $f=h\circ pr_1$ and $\beta=pr_2^*v$ where $M$ is as above and $v\in H^2(\PP^1(\CC),\ZZ)$ is the first Chern class of $\O(1)$. The real bundle $T_Y\oplus \RR^k$ admits an almost complex structure for some $k>0$, showing that $w_2(T_Y\oplus \RR^k)=w_2(T_Y)$ is the mod 2 restriction of $c_1(T_Y\oplus \RR^k)$, hence that $\beta_{\ZZ}w_2(T_Y)=0$ in $H^3(Y,\ZZ)$. This concludes since this characteristic class is the only obstruction to an orientable $6$-manifold carrying an almost complex structure (see \cite[pp.~559-560, especially Remark 1]{massey}).
\end{proof}

\begin{remark}
\label{rkconiveaus}
When $X$ is projective, Theorem \ref{main3} demonstrates that there is no topological obstruction to $\alpha$ having strong coniveau $\geq 1$ for $(l,n)=(3,3)$ or $(l,n)=(4,4)$.
There are however obstructions to $\alpha$ having strong coniveau $1$ coming from Hodge theory: it is necessary that $\alpha$ has Hodge coniveau $\geq 1$, in the sense that its image in $H^l(X,\CC)$ has no component of type $(l,0)$ or $(0,l)$ in the Hodge decomposition. Of course, this Hodge-theoretic obstruction is also an obstruction to $\alpha$ having coniveau $\geq 1$. We do not know of any obstructions to a coniveau $\geq 1$ class having strong coniveau $\geq 1$ for these values of $(l,n)$.
\end{remark}

\section{Rational coefficients}
\label{ratcoeffs}

We now provide examples of complex varieties for which the coniveau and strong coniveau filtrations for rational cohomology classes differ. By Deligne (see Theorem~\ref{Deligne}), this cannot occur for smooth proper varieties.

We recall that a morphism $f:X\to Y$ of equidimensional complex varieties is \textit{semismall} if $\dim(X\times_Y X)\leq \dim(X)$.

\subsection{A geometric construction}
\label{geometric}

Our examples are based on the following lemma.

\begin{lemma}
\label{negativeconstruction}
Fix $l\geq 2$ and write $l=2r+k-1$ with $r\in\NN$ and $k\in\{0,1\}$. 
There exist a rational smooth projective complex variety $S$ of dimension $l$, a smooth codimension~$r$ subvariety $\iota:D\hookrightarrow S$, a morphism of normal projective varieties $g:S\to \overline{S}$, a finite set $\overline{\iota}:\overline{D}\hookrightarrow\overline{S}$ such that $g^{-1}(\overline{D})=D$ and $g$ is an isomorphism above $\overline{S}\setminus \overline{D}$, and a nonzero class $\rho\in H^k(D,\QQ)$ such that $\iota_*\rho=0$ in $H^{l+1}(S,\QQ)$.
\end{lemma}
\begin{proof}
We first consider the case $k=0$.
 Let $P = \PP(\E)$ where $\E$ is the vector bundle $\O^2\oplus \O(1)^{r}$ over $\PP^{r-1}$. Then $P$ is of dimension $2r$, and the tautological bundle $M=\O_P(1)$ gives a morphism $G:P\to \PP^n$ which contracts exactly the subvariety $\PP(\O^2)\simeq \PP^1\times \PP^{r-1}$ to a $\PP^1$. Then let $S$ be a generic divisor in $|2M|$ which is smooth by the Bertini theorem. The morphism $G|_S:S \to \PP^n$ now contracts two disjoint copies $L_1,L_2$ of $\PP^{r-1}$ to two points. Let $\overline{S}$ be the normalization of the image of $G|_S$ with induced morphism $g:S\to\overline{S}$, and define $D=L_1\cup L_2$ and $\rho=[L_1]-[L_2]\in H^0(D,\QQ)$. The variety $S$ is a quadric bundle over $\PP^{r-1}$ of dimension $2r-1$, and it is rational as it contains a section ($L_1$ for instance).  Note that the morphism $G$ induced by $M$ is semismall. By the semismall version of the weak Lefschetz theorem (see \cite[Proposition 2.1.5]{dCMsemismall}), the restriction map $H^{2r-2}(P,\QQ)\to H^{2r-2}(S,\QQ)$ is an isomorphism. By Poincar\'e duality, so is the pushforward map $H^{2r}(S,\QQ)\to H^{2r+2}(P,\QQ)$. Clearly the class $i_*\rho$ maps to 0 by this map, so we conclude that $i_*\rho=0$, as we want.
 
 For the $k=1$ case, we use a similar construction. Let $V=\PP^2\times \PP^{r-1}$ and let $H_1$ and $H_2$ denote the two pullbacks from the hyperplane bundles on each factor. 
Let $P = \PP(\O \oplus \O(H_1+H_2)^{r})$ over $V$ and let $M=\O_P(1)$ denote the tautological bundle. Note that $P$ has dimension $2r+1$. The morphism $G:P\to \PP^n$ given by $M$ contracts exactly the codimension $r$ subvariety $W=\PP(\O)\simeq \PP^2\times \PP^{r-1}$ to a point.
Now let $S$ be a generic divisor in $|M+3H_1|$, which is smooth by the Bertini theorem. Note that $S$ is rational, since the projection $S\to V$ is generically a $\PP^{r-1}$\nobreakdash-bundle over $V$. Let $\overline{S}$ be the normalization of the image of $G|_S$. The induced morphism $g:S\to\overline{S}$ is birational and contracts exactly the locus $D=S \cap W$ to a point. The latter is a divisor of type $3H_1$ on $\PP^2 \times \PP^{r-1}$, thus isomorphic to $E\times \PP^{r-1}$, where $E$ is an elliptic curve. Hence there is a non-zero class $\rho\in H^1(D,\QQ)$. Since the morphism induced by $M+3H_1$ is semismall (it contracts $W\simeq \PP^2\times \PP^{r-1}$ to a $\PP^2$), the semismall version of the weak Lefschetz theorem  \cite[Proposition 2.1.5]{dCMsemismall} shows that $S$ has no odd degree cohomology. It follows that $i_*\rho=0$, as we wanted to show. 
 \end{proof}

\subsection{The kernel of local intersection forms} 
\label{dCM}

Lemma \ref{decompolemma} below is an application of the decomposition theorem of Beilinson, Bernstein, Deligne and Gabber \cite[Th\'eor\`eme~6.2.5]{BBDG}, as well as of a closely related theorem of de Cataldo and Migliorini \cite[Theorem~2.1.10]{dCM} which studies intersection forms on the homology of the fibers of a projective morphism with smooth total space.
For an overview of these topics, we refer to \cite{dCM} or Williamson's survey \cite{williamson}.

We use freely the theory of perverse sheaves \cite{BBDG} (see also the survey \cite{dCMsurvey}). If $X$ is a complex variety, we let
$D^b(X)$ be the bounded derived category of sheaves of $\QQ$-vector spaces on $X$, and $D^b_c(X)$ be the full subcategory of objects with constructible cohomology 
%coincides with bounded derived category of constructible sheaves of Q-vector spaces by a theorem of Nori
(see \cite[\S 1.5, \S 5.3]{dCMsurvey}). The triangulated category $D^b_c(X)$ may be endowed with the perverse $t$-structure (see \cite[\S 2.3]{dCMsurvey}). The heart of this $t$-structure is the abelian category $\Perv(X)$ of perverse sheaves on $X$.

We keep the notation of Lemma \ref{negativeconstruction}. 

\begin{lemma}
\label{decompolemma}
Let $T$ be a smooth projective variety of dimension $n:=l-1$, let $f:T\to S$ be a morphism, and define $E:=f^{-1}(D)$ with inclusion $j:E\hookrightarrow T$.
Consider the composition
\begin{equation}
\label{refinedintersection}
\psi:H_{n}(E,\QQ)\xrightarrow{j_*} H_{n}(T,\QQ)\simeq H^{n}(T,\QQ)\xrightarrow{j^*}H^{n}(E,\QQ),
\end{equation}
where the middle isomorphism stems from Poincar\'e duality.  Then
$$\ker \left(\psi:H_{n}(E,\QQ)\to H^{n}(E,\QQ)\right)\subset\ker\left((f|_E)_*:H_{n}(E,\QQ)\to H_{n}(D,\QQ)\right).$$
\end{lemma}

\begin{proof}
Let $\varepsilon\in H_{n}(E,\QQ)$ be such that $\psi(\varepsilon)=0$. We will show that ${(f|_E)_*\varepsilon=0}$.
To do so, we use the computation of $\ker(\psi)$ by de Cataldo and Migliorini \cite[Theorem~2.1.10]{dCM} in terms of an induced perverse filtration on $H_{n}(E,\QQ)$. The decomposition theorem \cite{BBDG} will then allow us to control this filtration.

Let $\omega_S=\QQ_S[2l]$, $\omega_T=\QQ_T[2n]$, $\omega_D$ and $\omega_E$ be the dualizing complexes of $S$,~$T$,~$D$ and~$E$.
There are natural isomorphisms (as in \cite[\S 3.4]{dCM}, see also \cite[\S 5.8]{dCMsurvey} for a formulary in constructible bounded derived categories):
\begin{equation*}
\begin{alignedat}{4}
H_{n}(E,\QQ)&=H^{-n}(E,\omega_E)=H^{-n}(E,j^!\omega_T)=H^{0}(E,j^!\QQ_T[n])\\
&=H^{0}(\overline{D},R((g\circ f)|_E)_*j^!\QQ_T[n])=H^{0}(\overline{D},\overline{\iota}^!R(g\circ f)_*\QQ_T[n]),
\end{alignedat}
\end{equation*}
 and similarly 
 \begin{equation*}
\begin{alignedat}{4}
H_{n}(D,\QQ)&=H^{-n}(D,\omega_D)=H^{-n}(D,\iota^!\omega_S)=H^{1}(D,\iota^!\QQ_S[l])\\
&=H^{1}(\overline{D},R(g|_D)_*\iota^!\QQ_S[l])=H^{1}(\overline{D},\overline{\iota}^!Rg_*\QQ_S[l]).
\end{alignedat}
\end{equation*}

As in \cite[\S\S4.2-4.3]{dCM}, endow $H_{n}(E,\QQ)$ 
with the increasing filtration induced by the perverse filtration of the complex $R(g\circ f)_*\QQ_T[n]$
in the following way:
$$H_{n,\leq s}(E,\QQ)=\im\left(H^{0}(\overline{D},\overline{\iota}^!\prescript{\perv}{}\tau_{\leq s}R(g\circ f)_*\QQ_T[n])\to H^{0}(\overline{D},\overline{\iota}^!R(g\circ f)_*\QQ_T[n])\right),$$
where the $\prescript{\perv}{}\tau_{\leq s}$ are the perverse truncation functors (see \cite[\S 2.3]{dCMsurvey}). 

One has $H_{n,\leq 0}(E,\QQ)=H_{n}(E,\QQ)$ by \cite[Lemma 4.3.6]{dCM}. Applying \cite[Theorem~2.1.10]{dCM} to $g\circ f$ with $a=b=0$ shows that $\ker(\psi)=H_{n,\leq -1}(E,\QQ)$, hence that $\varepsilon$ lifts to a class $\widetilde{\varepsilon}\in H^{0}(\overline{D},\overline{\iota}^!\prescript{\perv}{}\tau_{\leq -1}R(g\circ f)_*\QQ_T[n])$.

The morphism $Rf_*\omega_T\to\omega_S$ obtained by adjunction from the isomorphisms $f^!\omega_S\simeq \omega_T$ and $Rf_*\simeq Rf_!$ yields a morphism $\nu:R(g\circ f)_*\QQ_T[n]\to (Rg_*\QQ_S[l])[1]$ which induces
the pushforward $(f|_E)_*:H_{n}(E,\QQ)\to H_{n}(D,\QQ)$ (see \cite[\S 3.4]{dCM}). We deduce a commutative diagram whose vertical arrows are induced by $\nu$:
\begin{equation}
\label{diagpervers}
\begin{aligned}
\xymatrix
@R=0.3cm
{
H^{0}(\overline{D},\overline{\iota}^!\prescript{\perv}{}\tau_{\leq -1}R(g\circ f)_*\QQ_T[n])\ar[r]\ar^{}[d]&
H_{n}(E,\QQ)\ar^{(f|_E)_*}[d]  \\
H^{1}(\overline{D},\overline{\iota}^!\prescript{\perv}{}\tau_{\leq 0}Rg_*\QQ_S[l])\ar[r]& H_{n}(D,\QQ).
}
\end{aligned}
\end{equation}

The decomposition theorem \cite[Th\'eor\`eme 6.2.5]{BBDG} applied to $g\circ f$ shows that $\prescript{\perv}{}\tau_{\leq -1}R(g\circ f)_*\QQ_T[n]=\bigoplus_{s\leq -1} P_s[-s]$,
where $P_s$ is a direct sum of simple perverse sheaves.

The morphism $g:S\to\overline{S}$ is semismall in the sense that $\dim(S\times_{\overline{S}}S)\leq \dim(S)$ because $2\dim(D)=2r+2k-2\leq 2r+k-1=\dim(S)$. The particular shape taken by the decomposition theorem for semismall morphisms \cite[\S 1.7]{BorhoMacPherson} (see also \cite[Theorem 2.4]{williamson}) shows that $Rg_*\QQ_S[l]$ is a perverse sheaf (hence that  $\prescript{\perv}{}\tau_{\leq 0}Rg_*\QQ_S[l])=Rg_*\QQ_S[l]$), and that there exists an isomorphism $Rg_*\QQ_S[l]=\IC(\overline{S})$ if $k=0$ (resp. $Rg_*\QQ_S[l]=\IC(\overline{S})\oplus\overline{\iota}_*\QQ_{\overline{D}}$ if $k=1$). Here, we have denoted by $\IC(\overline{S})$ the intersection complex of $\overline{S}$, which is a simple perverse sheaf (see \cite[\S 3.8]{dCM}).

The morphism $\prescript{\perv}{}\tau_{\leq -1}\nu:\bigoplus_{s\leq -1} P_s[-s]\to (Rg_*\QQ_S[l])[1]$
vanishes on the direct summand $\bigoplus_{s\leq -2} P_s[-s]$ by \cite[D\'efinition 1.3.1~(i)]{BBDG} since perverse sheaves form the heart of a $t$-structure. The induced morphism 
$P_{-1}[1]\to\IC(\overline{S})[1]$ also vanishes since a morphism of simple perverse sheaves is either zero or an isomorphism (as in any abelian category), and since the support of $\IC(\overline{S})$ is equal to $\overline{S}$ whereas the supports of the simple factors of $P_{-1}$ cannot be equal to $\overline{S}$ as $g\circ f$ is not dominant.

The left vertical arrow of (\ref{diagpervers}) is
obtained by applying the functor $H^0(\overline{D}, \overline{\iota}^!(-))$ to $\prescript{\perv}{}\tau_{\leq -1}\nu$. The above shows that it vanishes if $k=0$, and that it is induced by
a morphism $H^0(\overline{D}, \overline{\iota}^!P_{-1}[1])\to H^0(\overline{D}, \overline{\iota}^!\overline{\iota}_*\QQ_{\overline{D}}[1])$
if $k=1$. The computation $H^0(\overline{D}, \overline{\iota}^!\overline{\iota}_*\QQ_{\overline{D}}[1])=H^1(\overline{D}, \overline{\iota}^!\overline{\iota}_*\QQ_{\overline{D}})=H^1(\overline{D},\QQ_{\overline{D}})=0$,
%The second equality is clear. Let $\eta:S^0\to\overline{S}$ be the complementary open immersion. For any F, there is a triangle \iota_!\iota^!F->F->\eta_*\eta^*F->. Note that \iota_!=\iota_* as proper. Apply it to F=\iota_*Q_{\overline{D}} so that \eta^*F=0. this gives an iso \iota_*\iota^!Q_{\overline{D}}=\iota_*Q_{\overline{D}}. Since \iota_* is fully faithful, \iota^!Q_{\overline{D}}=Q_{\overline{D}}.
shows that it vanishes in all cases. In particular, the image of $\widetilde{\varepsilon}$ by the left vertical arrow of (\ref{diagpervers}) is zero, and the commutativity of (\ref{diagpervers}) shows that $(f|_E)_*\varepsilon=0$.
\end{proof}

\begin{remark}
When $l=3$, Lemma \ref{decompolemma} follows from Mumford's theorem that the intersection matrix of the irreducible components of a contractible curve in a smooth projective surface is negative definite \cite[p.~6]{mumford}.
Indeed, this theorem, applied to the $1$-dimensional components of $E$, shows that $\psi:H_2(E,\QQ)\to H^2(E,\QQ)$ is an isomorphism (unless $E=T$ in which case Lemma \ref{decompolemma} is obvious because $\psi$ is an isomorphism).
This particular case of Lemma \ref{decompolemma} would be sufficient to prove Theorem \ref{main4} for $l=2c+1$.
\end{remark}

\subsection{Open varieties} 
\label{con1examples}

We still keep the notation of Lemma \ref{negativeconstruction}. Define $S^0:=S\setminus D$ and let $\gamma\in H^l(S^0,\QQ)$ be a lift of $\rho$ in the long exact sequence of the pair $(S,D)$:
\begin{equation}
\label{longrelative}
\dots\to H^l(S^0,\QQ)\to H^k(D,\QQ)\xrightarrow{\iota_*} H^{l+1}(S,\QQ)\to\dots
\end{equation}

\begin{lemma}
\label{open}
If $l\geq 3$, then $\gamma\in H^l(S^0,\QQ)$ has coniveau $\geq 1$ and strong coniveau~$0$.
\end{lemma}

\begin{proof}
Let $\widetilde{\gamma}$ be the image of $\gamma$ in $H^0(S^0,\mathcal{H}^l(\QQ))$ (see \S\ref{coniveautorsion}).
Since $l\geq 3$, $D$ has codimension $r\geq 2$ in $S$. Gersten's conjecture for Betti cohomology, proven by Bloch and Ogus, thus shows that $H^0(S^0,\mathcal{H}^l(\QQ))=H^0(S,\mathcal{H}^l(\QQ))$ (see \cite[Theorem~6.1]{blochogus}), and this group vanishes because  $S$ is rational (see \cite[Proposition~3.3~(i)]{colliotvoisin}). It follows that $\widetilde{\gamma}=0$ hence that $\gamma$ has coniveau $\geq 1$.

Assume for contradiction that the class $\gamma$ has strong coniveau~$\geq 1$. Then there exist a smooth complex variety $T^0$ of dimension $l-1$, a proper morphism $f^0:T^0\to S^0$ which we may assume to be projective by Chow's lemma, and a class $\delta\in H^{l-2}(T^0,\QQ)$ such that $f^0_*\delta=\gamma$ (that we may choose $T^0$ of dimension $l-1$ is explained in \S\ref{2filtrations}). 

Let $T$ be a smooth projective compactification of $T^0$ such that $f^0$ extends to a morphism $f:T\to S$. Define $E:=T\setminus T^0=f^{-1}(D)$ and let $j:E\hookrightarrow T$ be the inclusion.
The long exact sequences of $(S,D)$ and $(T,E)$ in Borel--Moore homology fit into a commutative exact diagram (see \cite[IX.2.1]{iversen})
%\cite[(1.2.4) and Example 2.3]{blochogus}
\begin{equation}
\label{BOpushfwd}
\begin{aligned}
\xymatrix
@R=0.3cm
{
\cdots\ar[r]&
H^{\BM}_l(T^0,\QQ)\ar^{}[r]\ar^{f^0_*}[d]&
H_{l-1}(E,\QQ)\ar^{(f|_E)_*}[d] 
\ar^{j_*}[r]& H_{l-1}(T,\QQ)\ar^{f_*}[d]
\ar[r]&\cdots \\
\cdots\ar[r]&
H^{\BM}_l(S^0,\QQ)\ar[r]& H_{l-1}(D,\QQ) 
\ar^{\iota_*}[r]& H_{l-1}(S,\QQ)
\ar[r]&\cdots
}
\end{aligned}
\end{equation}
whose bottom row identifies with (\ref{longrelative}) via Poincar\'e duality. Lemma \ref{decompolemma} shows that $\ker(j_*)\subset \ker((f|_E)_*)$ as subspaces of $H_{l-1}(E,\QQ)$.
Let $\varepsilon\in H_{l-1}(E,\QQ)$ be the image by the upper left horizontal arrow of (\ref{BOpushfwd}) of the class $\delta^\vee\in H^{\BM}_l(T^0,\QQ)$, Poincar\'e dual to $\delta$. The exactness of (\ref{BOpushfwd}) shows that $j_*\varepsilon=0$.
The commutativity of (\ref{BOpushfwd}) shows that $(f|_E)_*\varepsilon$ is the class $\rho^{\vee}\in H_{l-1}(D,\QQ)$,  Poincar\'e dual to $\rho$, which is nonzero. This contradicts the inclusion $\ker(j_*)\subset \ker((f|_E)_*)$.
\end{proof}

It is easy to deduce analogous examples for higher values of the coniveau.

\begin{theorem}
\label{main4}
For all $c\geq 1$ and $l\geq 2c+1$, there exists a smooth quasi-projective rational complex variety $X$ of dimension $l-c+1$
such that the inclusion $\widetilde N^cH^{l}(X,\QQ)\subset N^cH^{l}(X,\QQ)$ is strict.
\end{theorem}

\begin{proof}
By Lemma \ref{open}, we may find a smooth complex variety $S^0$ of dimension $l-2c+2$ and a class $\gamma\in H^{l-2c+2}(S^0,\QQ)$ which has coniveau $\geq 1$ and strong coniveau~$0$. Define
$X:= S^0\times \PP^{c-1}$ with projections $p:X\to S^0$ and $q:X\to \PP^{c-1}$, and consider the class $\alpha:=p^*\gamma\smile q^*\lambda$, where $\lambda\in H^{2c-2}(\PP^{c-1},\QQ)$ is the class of a point $t\in\PP^{c-1}$. Since $\alpha$ is the pushforward of $\gamma$, which has coniveau $\geq 1$, by the closed immersion $S^0\times\{t\}\to S^0\times \PP^{c-1}$, it has coniveau $\geq c$. Since $p_*\alpha=\gamma$ and $\gamma$ has strong coniveau $0$, we see that $\alpha$ has strong coniveau $<c$.
\end{proof}

\subsection{Singular varieties} 
\label{singularpar}

There is no strong coniveau filtration on the cohomology of a singular variety, as there do not exist pushforward morphisms associated with arbitrary proper morphisms of singular varieties. However, there exist variants of both the coniveau filtration and the strong coniveau filtration on the Borel--Moore homology of an arbitrary variety (in particular, on the homology of a proper variety).

Based on the examples of \S\S\ref{geometric}--\ref{con1examples}, we show that these variants do not allow to extend Deligne's Theorem \ref{Deligne} to arbitrary, not necessarily smooth, proper varieties.

\begin{theorem}
\label{singular}
For all $c\geq 1$ and $l\geq 3$
there exists a normal projective variety~$X$ of dimension~$l+c-1$ and a class $\zeta\in H_{l}(X,\QQ)$ with the following properties:
\begin{enumerate}[(i)]
\item One has $\zeta|_{X\setminus Z}=0$ in $H_{l}^{\BM}(X\setminus Z,\QQ)$ for some closed subset $Z\subset X$ of codimension $\geq c$ in $X$.
\item There do not exist a smooth proper variety $Y$ of dimension $\leq l-1$, a morphism $f:Y\to X$ and a class $\xi\in H_{l}(Y,\QQ)$ such that $f_*\xi=\zeta$.
\end{enumerate}
\end{theorem}

\begin{proof}
We use the notation of Lemmas \ref{negativeconstruction} and \ref{open}.
 The exact sequence of the pair $(\overline{S},\overline{D})$ 
in Borel--Moore homology \cite[IX.2.1]{iversen}
reads:
\begin{equation}
\label{respushfwd}
\begin{aligned}
\xymatrix
@R=0.3cm
@C=0.3cm
{
\cdots\ar[r]&
H_{l}(\overline{S},\QQ)\ar[r]& H^{\BM}_{l}(S^0,\QQ)
\ar^{}[r]& H_{l-1}(\overline{D},\QQ)=0.
}
\end{aligned}
\end{equation}
Let $\gamma^{\vee}\in H^{\BM}_{l}(S^0,\QQ)$ be Poincar\'e dual to the class $\gamma\in H^{l}(S^0,\QQ)$, 
and choose a lift $\eta\in H_{l}(\overline{S},\QQ)$ of $\gamma^{\vee}$ in 
(\ref{respushfwd}).
We define $X:=\overline{S}\times \PP^{c-1}$. 
Let $\zeta\in H_{l}(X,\QQ)$ be the pushforward of $\eta$ by the natural morphism $\overline{S}\times\{t\}\to \overline{S}\times \PP^{c-1}$, where $t\in\PP^{c-1}$ is a point.

Since $\gamma$ has coniveau $\geq 1$ by Lemma \ref{open}, one has $\eta|_{\overline{S}\setminus F}=0$ in $H_l^{\BM}(\overline{S}\setminus F,\QQ)$ for some proper closed subset $F\subset \overline{S}$. Taking $Z:=F\times\{t\}$ proves assertion (i).

Assume for contradiction that there exist a variety $Y$, a morphism $f$ and a class $\xi$ as in (ii). Define $g:=pr_1\circ f:Y\to \overline{S}$, let $Y^0:=g^{-1}(S^0)$ and let $g^0:=g|_{Y^0}:Y^0\to S^0$. Define $\xi^{\vee}\in H^*(Y,\QQ)$ to be the class Poincar\'e dual to $\xi$. Then the class $g^0_*(\xi^{\vee}|_{Y^0})$ is Poincar\'e dual to $g^0_*(\xi|_{Y^0})=(g_*\xi)|_{S^{0}}=((pr_1)_*\zeta)|_{S^{0}}=\eta|_{S^0}=\gamma^{\vee}$, hence equal to~$\gamma$. This contradicts that $\gamma$ has coniveau $0$ by Lemma \ref{open}, and concludes the proof.
\end{proof}

\section{Further questions}
\label{questions}

\subsection{Three-dimensional examples} Does there exist a smooth projective complex threefold $X$ such that the inclusion $\widetilde N^1H^3(X,\ZZ)\subset N^1H^3(X,\ZZ)$ is strict? In view of Remark \ref{rkconiveaus}, this would require an obstruction to having high strong coniveau which is not topological.

\subsection{Rationally connected examples} 
For a smooth projective rationally connected variety, it follows from \cite[Proposition 3.3 (i)]{colliotvoisin} that $N^1H^{l}(X,\ZZ)=H^{l}(X,\ZZ)$ for any $l>0$. 
Are there examples of such varieties such that the inclusion $\widetilde N^1H^l(X,\ZZ)\subseteq H^l(X,\ZZ)$ is strict? What about the case $l=3$? Such a variety could not be rational by Corollary \ref{rational}. This question was suggested to us by Claire Voisin.

\subsection{Positive results for threefolds.} Does $\widetilde N^1H^3(X,\ZZ)=N^1H^3(X,\ZZ)$ hold for some particular classes of threefolds, beyond Example \ref{exAJ}?
Voisin \cite[Theorem~0.2]{voisin} has very recently proven this equality modulo torsion when $X$ is a rationally connected threefold. 
Desingularizations of nodal quartic threefolds and the Artin--Mumford threefold give natural test cases (some of these rationally connected threefolds are known to have torsion in $H^3(X,\ZZ)$).

\subsection{Further discrepancy between coniveau and strong coniveau} Can one find a smooth projective complex variety $X$ and a class $\alpha\in H^l(X,\ZZ)$ that has coniveau $\geq c$ but strong coniveau $\leq c-2$? What about $c=2$ and $l=5$?

\subsection{Specialization of strong coniveau}
Suppose $f: \X\to T$ is a smooth projective family over a smooth connected curve $T$. If $\a\in H^l(\X,\ZZ)$,
and if $\a_t\in H^l(\X_t,\ZZ)$ has strong coniveau $\geq c$ for all $t\neq 0$, does $\a_0\in H^l(\X_0,\ZZ)$ have strong coniveau~$\geq c$? 
If this question had a positive answer, one could hope to construct cohomology classes for which coniveau and strong coniveau differ by degeneration arguments.

\subsection{Finite coefficients}
For a prime number $p$, and integers $c\geq 1$ and $l\geq 2c+1$, does there exist a smooth projective complex variety $X$ such that the inclusion $\widetilde N^cH^l(X,\ZZ/p)\subset N^cH^l(X,\ZZ/p)$ is strict? What about $p=2$, $c=1$ and $l=3$?

\bibliographystyle{myamsalpha}
\bibliography{coniveau}
\end{document}